\documentclass[12pt]{article}

\usepackage[margin=2.5cm, top=2.5cm, bottom=2.5cm]{geometry}

\usepackage[utf8]{inputenc} % allow utf-8 input
\usepackage[T1]{fontenc}    % use 8-bit T1 fonts

\usepackage{url}            % simple URL typesetting
\usepackage{booktabs}       % professional-quality tables
\usepackage{amsfonts}       % blackboard math symbols
\usepackage{nicefrac}       % compact symbols for 1/2, etc.
\usepackage{microtype}      % microtypography
  
\usepackage{amsmath,amssymb,amsthm,xcolor}
\usepackage{array}
\usepackage{float}

\usepackage{wrapfig}
\usepackage[hidelinks,hypertexnames=false]{hyperref}
\usepackage[capitalise]{cleveref}
\theoremstyle{definition}
\newtheorem{example}{Example}
\newtheorem{remark}{Remark}
\newtheorem{definition}{Definition}

\theoremstyle{plain}
\newtheorem{theorem}{Theorem}
\newtheorem{lemma}{Lemma}
\newtheorem{corollary}{Corollary}

\newtheorem{proposition}{Proposition}
\newtheorem{fact}[theorem]{Fact}
\usepackage{graphicx}
\usepackage{subcaption}
\usepackage{mdframed,framed}
\usepackage{lipsum}
\newmdtheoremenv{theo}{Theorem}
\usepackage{tikz,calc}
\usepackage{enumitem}
\usetikzlibrary{calc}
\allowdisplaybreaks
\usepackage{pgfplots}
\usetikzlibrary{intersections, pgfplots.fillbetween}
\usepackage{epstopdf}
\usepackage{algorithmic}
\usepackage{booktabs}
\usepackage{changepage}
\usepackage{centernot}
\usepackage{stmaryrd}
\usepackage[nottoc]{tocbibind}
\newcommand{\forae}{%
  \tikz[baseline={(forall.base)}]{
    \node[inner sep=0pt, outer sep=0pt] (forall) {$\forall$};
    \fill[white] (-0.08em,0.05ex) rectangle (0.08em,0.25ex);
  }%
}

\usepackage{empheq}
\usepackage{shadethm}

\newshadetheorem{assume}{Assumption}
\definecolor{shadethmcolor}{HTML}{FFFFFF}
\definecolor{shaderulecolor}{HTML}{000000}
\newcommand{\R}{\mathbb{R}}
\newcommand{\fg}{\mathfrak{g}}

\newcommand{\dom}{\mathrm{dom}}
\newcommand{\rank}{\mathrm{rank}\hspace{.3mm}}
\newcommand{\diag}{\mathrm{diag}}
\newcommand{\gph}{\mathrm{gph}}
\newcommand{\vp}{\varphi}
\newcommand{\cM}{M}
\newcommand{\epii}{\mathrm{epi}}
\newcommand{\rint}{\mathrm{int}\hspace*{.5mm}}
\newcommand{\cX}{\mathcal{X}}

\newcommand{\N}{\mathbb{N}}
\def\Im{\mathrm{Im}\hspace*{.5mm}}
\def\GL{\mathrm{GL}}
\def\p{^}

\newcommand{\rom}[1]{(\romannumeral #1)}
\title{\LARGE Subdifferentiation with symmetry}
\begin{document}

\author{\large C\'edric Josz\thanks{(\url{cj2638@columbia.edu}), IEOR, Columbia University, New York..}}

\date{}

\maketitle

\begin{center}
    \textbf{Abstract}
    \end{center}
    \vspace*{-3mm}
 \begin{adjustwidth}{0.2in}{0.2in}
 ~~~~ Given an objective function that is invariant under an action of a Lie group, we study how its subgradients relate to the orbits of the action. Our main finding is that they satisfy projection formulae analogous to those stemming from the Whitney and Verdier stratifications. If the function is definable in an o-minimal structure on the real field, then we also obtain an invariant variational stratification. On the application side, we derive a conservation law for subgradient dynamics under minimal assumptions. It can be used to detect instability in discrete subgradient dynamics. 
\end{adjustwidth} 
\vspace*{3mm}
\vspace*{3mm}

% 14P10 Semialgebraic set and related spaces
% 34A60 Differential inclusions 
% 49-XX Calculus of variations and optimal control, optimization
% 53-XX Differential geometry

\tableofcontents

\section{Introduction}
\label{sec:Introduction}
Symmetries are gaining significant interest in the machine learning community due to their prevalence in dictionary learning \cite{zhai2019understanding,zhai2020complete,qu2020geometric}, matrix factorization \cite{valavi2020revisiting,li2019symmetry,mishra2014fixed} and deep learning \cite{weiler2025equivariant,gluch2021noether,gens2014deep,godfrey2022symmetries,hu2019exploring}. The oral at NeurIPS last year by Marcotte \textit{et al.} \cite{marcotte2024abide} uses Lie bracket computations to determine the maximal number of independent conserved quantities. Their follow-up work \cite{marcotte2024keep} derives a conserved quantity for inertial dynamics in some neural networks. Zhao \textit{et al.} \cite{zhao2023symmetries} propose a conservation law for gradient dynamics and explore its relationship with flat minima. More recently still, Li \textit{et al.} \cite{ziyin2024implicit} use symmetries to analyze noisy gradient dynamics. On a different note, symmetry teleportation \cite{armenta2023neural,zhao2022symmetry,zhao2023improving} incorporates a teleportation step between two iterations in gradient descent. It consists in finding a point on the orbit of the current iterate with maximal gradient norm.

In contrast, symmetries have received little attention from the continuous optimization community, aside from convex relaxations \cite{riener-2013,wang2025real,josz2018lasserre,mcrae2024benign}.
%and Morse-Bott functions \cite{rebjock2024fast,rebjock2024fast2}.  
We propose to initiate this study from the perspective of variational analysis. Consider an extended real-valued function $f:\R^n\to \R\cup\{\infty\}$ that is invariant under an action $\theta:~G\times \R^n\to\R^n$ of a Lie group $G$ on $\R^n$, namely 
$$  \forall (g,x)\in G\times \R^n,\quad f(\theta(g,x))=f(x).$$
If $f$ is differentiable, then by the chain rule from differential geometry \cite[Proposition 3.6(b)]{lee2012smooth}, one has $$\forall (v,x)\in T_eG\times \R^n,\quad \langle \nabla f(x), d(\theta^{(x)}
)_e(v)\rangle=0,$$  where $\theta^{(x)}:~ G\ni g\mapsto \theta(g,x)\in \R^n$ and $d(\theta^{(x)})_e:T_eG\rightarrow \R^n$ is the differential of $\theta^{(x)}$ at the identity $e$ of $G$. The image of this differential is none other than the tangent space $T_xGx$ of the orbit $Gx=\theta^{(x)}(G)$ at $x$. This yields the intuitive projection formula $$P_{T_xGx}\nabla f(x)=0.$$ 
It is a key ingredient for obtaining a conserved quantity in gradient dynamics \cite{zhao2023symmetries}. 

Our objectives in this work are as follows. First, we seek to extend this formula to nonsmooth setting, namely $$P_{T_xGx}\partial f(x)=\{0\},$$ where $\partial f$ is the subdifferential from variational analysis. This ``orbital projection formula'' will allow us to generalize the conservation law to subgradient dynamics. Surprisingly, the chain rules of subdifferential calculus \cite[Theorem 10.6, Exercise 10.7]{rockafellar2009variational} are of no use here. 

Second, we wish to show that the formula is not too sensitive to perturbations, ideally 
$$\forall v\in\partial f(x+\epsilon),\quad |P_{T_xGx}(v)|=O(|\epsilon|),$$ 
where the $|\cdot|=\sqrt{\langle \cdot,\cdot\rangle}$ is the Euclidean norm. This ``perturbed orbital projection formula'' has algorithmic implications, in particular for the stability of the subgradient method \cite{josz2024sufficient}, as will be discussed in \cref{sec:Discrete subgradient dynamics}.

Third, we would like to break down the domain of the function into finitely many pieces on which the function is smooth, each of which is invariant under the action, and satisfies a perturbed projection formula. Namely, if $X \subseteq \R^n $ is such a piece, it should be that $\theta(G,X) \subseteq X$ and for any $\overline x \in X$, we have
$$ \forall v\in\partial f(y),\quad |P_{T_xX}(v)|=O(|x-y|)$$
for $x \in X$ and $y \in \R^n$ near $\overline x$. The perturbed projection formula plays a key role in the avoidance of saddle points with the subgradient method via perturbations \cite{bianchi2023stochastic,davis2025active}. The invariant stratification could thus help analyze first-order methods for functions with symmetries. 

In order to achieve our aims, we rely on the slice theorem of Koszul \cite{koszul1965lectures} and Palais \cite{palais1961existence} and stratification of definable sets by Loi \cite{le1998verdier}. We build on various other results from variational analysis, o-minimal structures, and differential geometry. Due to the interdisciplinary nature of this work, an extended background section is included, and the proofs of several known facts are included.

As a reward for our efforts, we obtain a conserved quantity in subgradient dynamics $$\forae t>0,~x'(t)\in -\overline{\partial}f(x(t))$$ where $\forae$ means for almost every and $\overline{\partial}f$ is the Clarke subdifferential. This holds under minimal assumptions on the objective function $f$, provided that the action is linear. Unlike in previous works in machine learning \cite{marcotte2024abide,zhao2023symmetries,ziyin2024implicit}, the objective need not be smooth nor real-valued. This enables its application to important problems like ReLu neural networks and nonnegative matrix factorization. Also, we propose two closed-form expressions for the conserved quantity, one of which is
$$C(x) = P_{\mathrm{s}(\fg)}(xx\p T)$$
where $\mathrm{s}(\fg)$ denotes the symmetric elements of the Lie algebra $\fg$. We use this quantity to devise a sufficient condition for instability of the subgradient method with constant step size. By the way, most of the results extend to conservative fields, proposed by Bolte and Pauwels \cite{bolte2020conservative}.

This paper is organized as follows. \cref{sec:Background} provides background material on the different branches of mathematics used in later sections. The orbital projection formulae are obtained in \cref{{sec:orbit_proj}}. An invariant variational stratification is then derived in \cref{sec:invariant}. A conservation law is obtained for subgradient dynamics in \cref{sec:Conservation law} and several examples are given. Finally, the orbital projection formulae and the conserved quantity are used to establish instability of discrete subgradient dynamics in \cref{sec:Discrete subgradient dynamics}.

\section{Background}
\label{sec:Background}

We will borrow notions from variational analysis \cite{rockafellar2009variational}, %differential inclusions \cite{aubin1984differential}, 
o-minimal structures \cite{van1998tame}, and differential geometry \cite{lee2012smooth}. This section is intended for a broad audience. As usual,
\begin{gather*}
     \mathbb{N} = \{0,1,\hdots\}, ~~~ \mathbb{N}^* = \{1,2,\hdots\},~~~ \mathbb{R}_+ = [0,\infty), \\
    \mathbb{R}^* = \mathbb{R} \setminus \{0\}, ~~~ \R_+^* = (0,\infty),~~~\overline{\R} = \mathbb{R} \cup \{\infty\}.
\end{gather*}
% Given two integers $a \leq  b$, we use the notation 
% \begin{equation*}
%     \llbracket a,b \rrbracket = \{a,a+1,\hdots,b\}.    
% \end{equation*}
Also, $B_r(x)$ denotes the closed ball of radius $r$ centered at $x$. We consider neighborhoods to be open sets, following Lee \cite{lee2012smooth}. Finally, the sign of a real number $t$ is defined by 
 \begin{equation*}
     \mathrm{sign}(t) = 
     \left\{
     \begin{array}{cl}
          t/|t| & \text{if} ~ t\neq 0, \\
          \left[-1,1\right] & \text{else.} 
     \end{array}
     \right.
\end{equation*}
Let $\|\cdot\|_F$ and $\|\cdot\|_1$ respectively denote the Frobenius norm and the entrywise $\ell_1$-norm.

\subsection{Variational analysis}
\label{subsec:Variational analysis}
Variational analysis is concerned with the study of extrema and generalized differentiation, for which an indisputable reference is the book of Rockafellar and Wets \cite{rockafellar2009variational}. Below, we expose the notions used in this manuscript.

%is proper if $f(x) > -\infty$ for all $x \in \mathbb{R}^n$ and
The effective domain and the graph of an extended real-valued function $f:\mathbb{R}^n\rightarrow\overline{\mathbb{R}}$ are respectively defined by
\begin{gather*}
    \dom f = \{ x\in \mathbb{R}^n : f(x) < \infty\}, \\
    \mathrm{gph}f=\{ (x,t)\in \mathbb{R}^n\times \R : f(x)=t\}.
\end{gather*}
%is nonempty.

A prime example of an extended real-valued function is the indicator of a set $S \subseteq \mathbb{R}^n$, defined by 
\begin{equation*}
    \delta_S(x) = \left\{
    \begin{array}{cl}
        0 & \text{if} ~ x \in S, \\
        \infty & \text{if} ~ x \notin S.
    \end{array}
    \right.
\end{equation*}

A function $f:\mathbb{R}^n\rightarrow\overline{\mathbb{R}}$ is lower semicontinuous at $\overline{x} \in \dom f$ if $\liminf_{x\rightarrow \overline{x}} f(x) \geq   f(\overline{x})$ \cite[Definition 1.5]{rockafellar2009variational}. It is lower semicontinuous if it is so at every point in its domain. This is equivalent to requiring that the epigraph
\begin{equation*}
\mathrm{epi}f = \{ (x,t) \in \mathbb{R}^n \times \mathbb{R} : f(x) \leq  t \}
\end{equation*}
is closed \cite[Theorem 1.6]{rockafellar2009variational}. 

For a set-valued mapping $F:\mathbb{R}^n \rightrightarrows \mathbb{R}^m$, the domain and the graph are respectively defined by
\begin{gather*}
      \dom F=\{x\in \R^n:~F(x)\neq \emptyset \} ,\\
      \mathrm{gph}F = \{ (x,y)\in \mathbb{R}^n\times \R^m : F(x)\ni y \}.
\end{gather*}
The preimage is given by $F^{-1}(y) = \{ x\in \mathbb{R}^n : F(x) \ni y \}$. The graph of a function $F:A\to B$ where $A\subseteq \R^n$ and $B\subseteq \R^m$ is defined by $\mathrm{gph}(\widetilde F)$, where $\widetilde F:\R^n\rightrightarrows \R^m$ is such that $\widetilde F(x)=\{F(x)\}$ for $x\in A$ and $\widetilde F(x)=\emptyset$ for $x\notin A$. A set-valued mapping $F$ is said to be locally bounded if for all $x\in \dom  F$, there exists a neighborhood $U$ of $x$ in $\R\p n$ such that $F(U)$ is bounded.

Given $C\subseteq \R^n$, the horizon cone \cite[Definition 3.3]{rockafellar2009variational} is defined by
\[     C^\infty=\{v\in \R^n:~\exists t_k\searrow 0,~x_k\in C,~~t_kx_k\to v      \}      \]
if $C\neq\emptyset$ and $C^\infty= \{0\}$ otherwise.
Given $C \subseteq \mathbb{R}^n$ and $\overline{x} \in C$, the tangent cone \cite[Definition 6.1]{rockafellar2009variational} is defined by
\begin{gather*}
     T_C(\overline x)=\{v\in \R^n:~\exists  x_k\xrightarrow[C]{} \overline{x}~,~\tau_k\searrow 0: (x^k-\overline x)/\tau^k\to v\}
\end{gather*}
where $x_k \xrightarrow[C]{} \overline{x}$ is a shorthand for $x_k \rightarrow \overline{x}$ and $x_k \in C$.
The regular normal cone, normal cone \cite[Definition 6.3]{rockafellar2009variational}, and convexified normal cone \cite[6(19) p. 225]{rockafellar2009variational} are respectively defined by
\begin{gather*}
    \widehat{N}_C (\overline{x}) = \{ v \in \mathbb{R}^n : \langle v , x - \overline{x}\rangle \leq  o(|x-\overline{x}|)~\text{for}~x\in C~\text{near}~\overline{x} \}, \\ 
    N_C(\overline{x}) = \{ v \in \mathbb{R}^n : \exists x_k \xrightarrow[C]{} \overline{x}~\text{and}~ \exists v_k \rightarrow v~\text{with}~v_k \in \widehat{N}_C (x_k) \}, \\
    \overline{N}_C (\overline{x}) = \overline{\mathrm{co}} N_C (\overline{x}),
\end{gather*}
where ${\mathrm{co}}$ denotes the convex hull, and $\overline{\mathrm{co}}$ its closure. Explicitly, the $o$ means that 
\begin{equation*}
    \limsup_{\scriptsize \begin{array}{c}x \xrightarrow[C]{} \overline{x}\\ x\neq \overline{x}\end{array}} \frac{\langle v , x - \overline{x}\rangle}{|x-\overline{x}|} \leq  0.
\end{equation*}
The orthogonal and polar set of $S\subseteq \R^n$ are respectively defined by 
\begin{align*}
    S^\perp =\{v\in \R^n:~\forall w\in S,~\langle v,w\rangle = 0\}, \\
    S^* =\{v\in \R^n:~\forall w\in S,~\langle v,w\rangle\leq  0\}.
\end{align*}
By \cite[Theorem 6.28]{rockafellar2009variational}, the relationship $ \widehat{N}_C (\overline{x})=T_C(\overline{x})^*$ holds. A set $C \subseteq \mathbb{R}^n$ is regular \cite[Definition 6.4]{rockafellar2009variational} at one of its points $\overline{x}$ if it is locally closed and $\widehat{N}_C(\overline{x}) = N_C(\overline{x})$.

Given $f:\mathbb{R}^n\rightarrow \overline{\mathbb{R}}$ and a point $\overline{x}\in\mathbb{R}^n$ where $f(\overline{x})$ is finite, the regular subdifferential, subdifferential, horizon subdifferential \cite[Definition 8.3]{rockafellar2009variational}, and Clarke subdifferential of $f$ at $\overline{x}$ \cite[Definition 4.1]{drusvyatskiy2015curves} are respectively given by
\begin{gather*}
    \widehat{\partial} f (\overline{x}) = \{ v \in \mathbb{R}^n : f(x) \geq   f(\overline{x}) + \langle v , x - \overline{x} \rangle + o(|x-\overline{x}|) ~\text{near}~ \overline{x} \}, \\
    \partial f(\overline{x}) = \{ v \in \mathbb{R}^n : \exists (x_k,v_k)\in \gph\hspace*{.5mm}\widehat{\partial} f: (x_k, f(x_k), v_k)\rightarrow(\overline{x}, f(\overline{x}), v) \}, \\[1mm]
    \partial^\infty f(\overline{x}) = \{ v \in \mathbb{R}^n : \exists (x_k,v_k)\in \gph\hspace*{.5mm}\widehat{\partial} f: \exists \tau_k \searrow 0: (x_k, f(x_k), \tau_kv_k)\rightarrow(\overline{x}, f(\overline{x}), v) \}, \\[2mm]
    \overline{\partial} f(\overline{x}) = \overline{\mathrm{co}} [\partial f(\overline{x}) + \partial^\infty f(\overline{x})].
\end{gather*}
If $f(\overline{x})$ is not finite, then the subdifferentials are empty. The $o$ means that
\begin{equation*}
    \liminf_{\scriptsize \begin{array}{c}x \rightarrow \overline{x} \\ x\neq \overline{x}\end{array}} \frac{f(x) - f(\overline{x}) - \langle v , x - \overline{x} \rangle}{|x-\overline{x}|} \geq   0.
\end{equation*}

A function $f:\mathbb{R}^n \rightarrow \overline{\mathbb{R}}$ is regular \cite[Definition 7.25]{rockafellar2009variational} at $\overline{x}$ if $f(\overline{x})$ is finite and $\mathrm{epi}f$ is regular at $(\overline{x},f(\overline{x}))$ as a subset of $\mathbb{R}^{n+1}$. The normal cones and subdifferentials are related \cite[Theorem 9.10]{rockafellar2009variational} at any point $\overline{x}$ where $f$ is finite:
\begin{align*}
    \widehat{\partial} f (\overline{x}) &= \{ v \in \mathbb{R}^n : (v,-1) \in \widehat{N}_{\mathrm{epi}f} (\overline{x},f(\overline{x})) \}, \\
    \partial f (\overline{x}) & = \{ v \in \mathbb{R}^n : (v,-1) \in N_{\mathrm{epi}f} (\overline{x},f(\overline{x})) \}, \\
    \partial^\infty f (\overline{x}) & \subseteq \{ v \in \mathbb{R}^n : (v,0) \in N_{\mathrm{epi}f} (\overline{x},f(\overline{x})) \}, \\
    \overline{\partial} f (\overline{x}) & \supseteq \{ v \in \mathbb{R}^n : (v,-1) \in \overline{N}_{\mathrm{epi}f} (\overline{x},f(\overline{x})) \},
\end{align*}
where equality holds in the inclusions if $f$ is locally lower semicontinuous at $\overline{x}$. 

Following Bolte and Pauwels \cite{bolte2020conservative} and the extension of Josz \textit{et al.} \cite{josz2024prox}, given $f:\mathbb{R}^n\rightarrow \overline{\mathbb{R}}$ that is locally Lipshitz continuous on its domain, a set-valued mapping $D:\mathbb{R}^n \rightrightarrows \mathbb{R}^n$ is a conservative field for $f$ if it has closed graph, $\dom  D \subseteq \dom  f$, and for any absolutely continuous function $x:[0,1]\rightarrow\dom  D$, we have $$ \forae t\in (0,1),~\forall v\in D(x(t)),\quad (f\circ x)'(t) = \langle v , x'(t)\rangle.$$
 
A function $f:\mathbb{R}^n\rightarrow\overline{\mathbb{R}}$ is Lipschitz continuous near $\overline{x} \in \mathbb{R}^n$ if there exists $L>0$ such that $|f(x)-f(y)|\leq  L|x-y|$ for $x,y\in \dom  f$ near $\overline{x}$. It is locally Lipschitz on $X\subseteq \mathbb{R}^n$ if it is so at every point of $X$. If $f$ is Lipschitz continuous near $\overline{x}$, then let
\begin{equation*}
    \overline{\nabla} f(\overline{x}) = \{ v \in \mathbb{R}^n : \exists x_k \xrightarrow[D]{} \overline{x} ~\text{with}~ \nabla f(x_k) \rightarrow v\}
\end{equation*}
where $D$ are the differentiable points of $f$. Finally, given a point $x \in \mathbb{R}^n$, we use 
\begin{equation*}
    d(x,S) = \inf \{|x-y|: y \in S\}
~~~\text{and}~~~
    P_S(x) = \arg\min \{ |x-y| : y \in S\}
\end{equation*}
to denote the distance to and projection on a subset $S\subseteq \mathbb{R}^n$ respectively.  

\subsection{O-minimal structures}
\label{subsec:O-minimal structures}

O-minimal structures (short for order-minimal) were originally considered by van den Dries, Pillay, and Steinhorn \cite{van1984remarks,pillay1986definable}. They are founded on the observation that many properties of semi-algebraic sets can be deduced from a few simple axioms \cite{van1998tame}. Recall that a subset $A$ of $\mathbb{R}^n$ is semi-algebraic \cite{bochnak2013real} if it is a finite union of basic semi-algebraic sets, which are of the form 
    \begin{equation*}
        \{ x \in \mathbb{R}^n : f_1(x) > 0 , \hdots , f_p(x) > 0, f_{p+1}(x) = 0, \hdots , f_q(x) = 0\}
    \end{equation*}
    where $f_1,\hdots,f_q$ are polynomials with real coefficients. We adopt \cite[Definition p. 503-506]{van1996geometric} below.
    
\begin{definition}
\label{def:o-minimal}
An o-minimal structure on the real field is a sequence $S = (S_k)_{k \in \mathbb{N}}$ such that for all $k \in \mathbb{N}$:\vspace{-2mm}
\begin{enumerate}
    \item $S_k$ is a boolean algebra of subsets of $\mathbb{R}^k$, with $\mathbb{R}^k \in S_k$;\vspace{-2mm}
    \item $S_k$ contains the diagonal $\{(x_1,\hdots,x_k) \in \mathbb{R}^k : x_i = x_j\}$ for $1\leq  i<j \leq  k$;\vspace{-2mm}
    \item If $A\in S_k$, then $A\times \mathbb{R}$ and $\mathbb{R}\times A$ belong to $S_{k+1}$;\vspace{-2mm}
    \item If $A \in S_{k+1}$ and $\pi:\mathbb{R}^{k+1}\rightarrow\mathbb{R}^k$ is the projection onto the first $k$ coordinates, then $\pi(A) \in S_k$;\vspace{-2mm}
    \item $S_3$ contains the graphs of addition and multiplication;\vspace{-2mm}
    \item $S_1$ consists exactly of the finite unions of open intervals and singletons. 
\end{enumerate}
\end{definition}

A subset $A$ of $\mathbb{R}^n$ is definable in an o-minimal structure $(S_k)_{k\in\mathbb{N}}$ if $A \in S_k$ for some $k\in\mathbb N$. A subset of a Euclidean space $E$ of dimension $n$ is a definable if its image via isomorphism (i.e. an invertible linear map) from $E$ to $\mathbb{R}^n$ is definable in an o-minimal structure. Given two isomorphisms $F,G:E\to\R^n$, $GF^{-1}:\R^n\to \R^n$ is definable (in the classical sense). Thus, for any $S\subseteq E$, $F(S)$ is definable if and only if $G(S)=GF^{-1}(F(S))$ is definable.  
\begin{definition}
    A function $f:\mathbb{R}^n\rightarrow\overline{\mathbb{R}}$ is definable in an o-minimal structure if $\gph f$ is definable in that structure.
\end{definition}

Similarly, functions from $\R^n$ to $\R^m$ and set-valued mappings 
are definable if their graphs are. Examples of o-minimal structures include the real field with constants, whose definable sets are the semi-algebraic sets (by Tarski-Seidenberg \cite{tarski1951decision,seidenberg1954new}), the real field with restricted analytic functions, whose definable sets are the globally subanalytic sets (by Gabrielov \cite{gabrielov1968projections,van1986generalization}), the real field with the exponential function (by Wilkie \cite{wilkie1996model}), the real field with the exponential and restricted analytic functions (by van den Dries, Macintyre, and Marker \cite{van1994elementary}), the real field with restricted analytic and real power functions (by Miller \cite{miller1994expansions}), and the real field with convergent generalized power series (by van den Dries and Speissegger \cite{van1998real}). Note that there is no largest o-minimal structure \cite{rolin2003quasianalytic}. Throughout this paper, we fix an arbitrary o-minimal structure $(S_k)_{k\in\mathbb{N}}$.

A key property of univariate definable functions is that they satisfy the monotonicity theorem \cite[4.1]{van1996geometric}. It states that on bounded open intervals, for any $p\in \mathbb{N}$ there exist finitely many open subintervals where the function is $C^p$ and either constant or strictly monotone. This allows for a short proof of the \L{}ojasiewicz inequality \cite[Theorem 1.14]{pham2016genericity}. It asserts that if $f:\mathbb{R}^n\rightarrow\mathbb{R}$ is a continuous definable function with $f^{-1}(0) \neq \emptyset$ and $X \subseteq \mathbb{R}^n$ is a bounded set, then there exists an increasing definable diffeomorphism $\varphi:\mathbb{R}_+\rightarrow \mathbb{R}_+$ such that
\begin{equation*}
    \forall x \in X, ~~~ |f(x)| \geq \varphi(d(x,f^{-1}(0))).
\end{equation*}
By growth dichotomy, univariate definable functions are polynomially bounded if and only if the exponential function is not definable \cite{miller1994exponentiation}. In that case, one can then find $\theta>0$ such that $\varphi(t) \geq t^\theta$. We will draw inspiration from the proof of the \L{}ojasiewicz inequality when deriving a necessary condition for stability (see Theorem \ref{thm:unstable}). 

The extension of the monotonicity theorem to multivariate functions is the cell decomposition theorem \cite[4.2]{van1996geometric}, which plays a fundamental role. It implies that definable sets can be stratified in a particularly nice way, in that they can be broken up into finitely many smooth pieces that fit together nicely. This fact naturally transposes to definable functions. We will elaborate on this after a brief introduction to differential geometry.

\subsection{Differential geometry}
\label{subsec:Differential geometry}

Our starting point is a smooth manifold $M$, that is, a topological manifold equipped with a smooth structure. (By smooth, we mean $C^k$ for some fixed $k\in \{1,2,\hdots,\infty\}$). A topological manifold is a locally Euclidean (of constant dimension) second-countable Hausdorff topological space. In contrast to the branch of optimization dealing with optimization on smooth manifolds \cite{boumal2023introduction}, our variable will lie in a Euclidean space as usual.

The smooth structure enables one to define smooth maps between two manifolds $M,N$ as well as the tangent space $T_pM$ at a point $p \in M$. Tangent vectors $v\in T_pM$ are linear maps $v:C^k(M)\to \R$ such that $v(fg)=v(f)g+fv(g)$ for $f,g\in C^k(M)$. Tangent vectors can also be defined using an equivalence relation on the set of all smooth curves $\gamma:J\rightarrow M$ where $J$ is an interval of $\mathbb{R}$ containing $0$ and $\gamma(0) = p$ \cite[p. 71]{lee2012smooth}. Two such curves $\gamma_1:J_1\rightarrow M$ and $\gamma_2:J_2\rightarrow M$ are equivalent if $(f\circ \gamma_1)'(0) = (f\circ \gamma_2)'(0)$ for any smooth real-valued function defined in a neighborhood of $p$. The tangent space is then the set of equivalence classes.  

The differential of a smooth map $F:M\rightarrow N$ at $p\in M$ is the linear map $dF_p: T_pM \rightarrow T_{F(p)} N$ such that $dF_p(v)(f\circ F)=v(f)$ for all $v\in T_pM$ and $f\in C^k(N)$, The rank of $F$ at $p$ is the rank of $dF_p$, namely the dimension of the image of $dF_p$. A map $F:M\rightarrow N$ between two smooth manifolds $M,N$ is a smooth immersion (respectively submersion) if it is smooth and $dF_p$ is injective (respectively surjective) for all $p\in M$. It is a smooth embedding if it a smooth immersion and a topological embedding, i.e., a homeomorphism onto its image $F(M) \subseteq N$ in the subspace topology.  

The differential enables one to define notions of submanifolds, which arise prominently in the study of symmetries. An embedded submanifold of $M$ is a subset $S \subseteq M$ that is a manifold in the subspace topology, endowed with a smooth structure with respect to which the inclusion map $S \hookrightarrow M$ is a smooth embedding (the inclusion map is defined by $S \ni x\mapsto x\in M$). An immersed submanifold is a subset $S \subseteq M$ endowed with a topology (not necessarily the subspace topology) with respect to which it is a topological manifold, and a smooth structure with respect to which the inclusion map $S\hookrightarrow M$ is a smooth immersion. From the definition, one sees that embedded submanifolds are immersed manifolds, but the converse is false, as illustrated by the figure eight \cite[Example 4.19]{lee2012smooth}. Embedded submanifolds can be expressed locally as level sets of smooth submersions \cite[Proposition 5.16]{lee2012smooth}, which is how they are often defined in $\mathbb{R}^n$ \cite[Example 6.8]{rockafellar2009variational}.

Suppose $S$ is an immersed submanifold of a smooth manifold $M$. Since inclusion map $\iota:S\to M$ is a smooth immersion, its differential $d\iota_p:T_pS\to T_pM$ is injective for all $p\in S$. One may thus view $T_pS$ as a subspace of $T_pM$ via the identification $T_pS\cong d\iota_p(T_pS)$. The following characterization is helpful \cite[Proposition 5.35]{lee2012smooth}. A vector $v\in T_pM$ is in $T_pS$ if and only if there is a smooth curve $\gamma:J\to M$ whose image is contained in $S$, and which is also a smooth as a map into $S$, such that $0\in J$, $\gamma(0)=p$, and $\gamma'(0)=v$. When $M=\R^n$, $T_pM\cong \R^n$. Hence, $T_pS$ can simply be viewed as a subset of $\R^n$. In that case, one can define the normal space $N_pS=(T_pS)^\perp$. When $S$ is an embedded submanifold of $\R^n$, then the tangent space and the normal space agree with the tangent cone and the normal cones from variational analysis, respectively, namely $T_S(p)=T_pS$ and $\widehat{N}_S (p) = N_S (p) = N_{p} S$ \cite[Example 6.8]{rockafellar2009variational}.

\subsubsection{Actions}
\label{subsubsec:Actions}

Our second object of interest is a group $G$, that is, a set equipped with a binary operation that is associative, has an identity element $e$, and such that every element has an inverse element. It enables us to define an action (see for e.g., \cite[p. 161]{lee2012smooth}).

\begin{definition}
\label{def:action}
An action of a group $G$ on a set $M$ is a map $G \times M \rightarrow M$ such that
\begin{enumerate}[label=\rm{(\roman{*})}]
    \item $\forall g_1,g_2\in G, ~ \forall x\in M, ~ g_1 (g_2x) = (g_1g_2)  x$,
    \item $\forall x\in M, ~ ex = x$.
\end{enumerate}
% \begin{itemize}
%     \item[--] $\forall g_1,g_2\in G, ~ \forall x\in M, ~ g_1 (g_2x) = (g_1g_2)  x$,
%     \item[--] $\forall x\in M, ~ ex = x$.
% \end{itemize}
\end{definition}

We say that $G$ acts on $M$ if such an action exists, and that $G$ acts trivially on $M$ if $gx = x$ for all $(g,x)\in G \times M$. One also says that $M$ is a $G$-space to mean that $G$ acts on $M$, and that is it is a homogeneous $G$-space if the action is transitive. This mean that for all $x,y\in M$, there exists $g\in G$ such that $gx = y$. Actions of interest are often continuous and even smooth, which calls for combining the above notions. A topological group $G$ is a topological manifold and a group such that the binary operation and inversion are continuous. A Lie group $G$ is a smooth manifold and a group whose operations are smooth. In this work, we identify the Lie algebra $\fg$ with $T_eG$, following \cite[Theorem 8.37]{lee2012smooth}. An action of a topological (respectively Lie) group $G$ on a topological (resp. smooth) manifold $M$ is then continuous (resp. smooth) if the defining map $G \times M \rightarrow M$ is continuous (resp. smooth). In that case, we say that $G$ acts continuously (resp. smoothly) on $M$.

To go on, we need some more notions. A Lie group homomorphism is a smooth map between Lie groups that is also a group morphism\footnote{In fact, continuous group morphisms between Lie groups are smooth \cite[20-11 (b) p. 538]{lee2012smooth}. This implies that given a topology on $G$, there is only one smooth structure that makes $G$ into a Lie group.}.  A Lie subgroup of a Lie group $G$ is a subgroup of $G$ is endowed with a topology and smooth structure making into a Lie group and an immersed submanifold. A nice feature of this definition is that the image of a Lie group via a Lie group homomorphism is a Lie subgroup \cite[Proposition 7.17]{lee2012smooth}. A major result on Lie groups is the closed subgroup theorem \cite[Theorem 20.12]{lee2012smooth}, which asserts that every subgroup of a Lie group that is topologically closed is an embedded Lie subgroup. This becomes relevant when we consider linear actions, prevelant in optimization applications \cite{kunin2020neural,gluch2021noether,marcotte2024abide}, as follows (see \cite[p. 170]{lee2012smooth}).

\begin{definition}
\label{def:linear}
An action of a group $G$ on a vector space $V$ is linear if $V \ni x \mapsto gx\in V$ is linear for all $g \in G$. 
\end{definition}

Let $\mathrm{GL}(V)$ denote the invertible linear maps from $V$ to itself. A smooth action of a Lie group $G$ on a finite-dimensional vector space $V$ is linear if and only if it is of the form $(g,x) \mapsto \rho(g)x$ for some Lie group homomorphism $\rho: G \rightarrow \mathrm{GL}(V)$ \cite[Proposition 7.37]{lee2012smooth}, called a representation of $G$. At the same time, $\mathrm{GL}(V)$ is isomorphic to the general linear group $\mathrm{GL}(n,\mathbb{R})$, i.e., the set of invertible $n \times n$ matrices with real entries. Without loss of generality, we may thus always assume smooth linear actions to be in the form $ G \times \mathbb{R}^n \ni(g,x) \mapsto gx \in \mathbb{R}^n$ where $G$ is a Lie subgroup of $\mathrm{GL}(n,\mathbb{R})$ and $gx$ is the matrix vector multiplication. This is called the natural action of $G$ on $\R^n$ \cite[Example 7.22(b)]{lee2012smooth}. Actions in turn enable us to define invariance. 

\begin{definition}
\label{def:invariant}
A function $f:M\rightarrow N$ between sets $M$ and $N$ is invariant under an action of a group $G$ on $M$ if $f(gx) = f(x)$ for all $g\in G$ and $x\in M$.
\end{definition}

For brevity, we will sometimes say that $f$ is $G$-invariant to mean that $f$ is invariant under a smooth action of a Lie group $G$ on its domain (the starting set in the definition of a function). A related notion is as follows. It is useful for establishing that a map has constant rank.

\begin{definition}
\label{def:equivariant}
A function $f:M\rightarrow N$ between sets $M,N$ is equivariant with respect to an action of a group $G$ on $M$ and $N$ if $f(gx) = gf(x)$ for all $g\in G$ and $x\in M$.
\end{definition}

We next turn our attention to orbits.

\subsubsection{Orbits}
\label{subsubsec:Orbits}

Suppose $\theta: G \times M \rightarrow M$ is an action of a group $G$ on a set $M$. We introduce some standard vocabulary.
\begin{itemize}
    \item[--] The orbit of a point $x\in M$ is the set $Gx=\{ gx : g\in G\}$.
    \item[--]  The isotropy group of a point $x\in M$ is the set $G_x = \{g\in G : gx = x\}$.
    \item[--] The action is free if $G_x=\{e\}$ for all $x\in M$.
    \item[--] The action is free at $x \in M$ if $G_x = \{e\}$.
\end{itemize}
A map between topological spaces is proper if preimages of compact sets are compact. Suppose $G$ acts continuously on $M$.
\begin{itemize}
    \item[--] The action is proper if $G \times M \ni(g,x) \mapsto (gx,x)\in M \times M$ is proper.
    \item[--] The action is proper near $x\in M$ if there exists a neighborhood $U$ of $x$ such that $\{ g \in G : U \cap g U \neq \emptyset \}$ has compact closure.
\end{itemize}

The reason why we are interested in local properness is because the global condition is generally too strong: it implies a compact isotropy group at every point, which fails for natural actions of noncompact Lie subgroups of $\mathrm{GL}(n,\R)$ since $G_0 = G$. 

If an action is proper near $x \in M$, then it is proper near any point in its orbit. Indeed, if $y = hx$ for some $h \in G$, then $V = hU$ is a neighborhood of $y$ since $x\mapsto hx$ is a homeomorphism. Also, for all $g \in G$, $V \cap gV \neq \emptyset$ if and only if $U \cap h^{-1}gh U \neq \emptyset$, so that $h^{-1}gh$ has compact closure. One concludes by noting that $g \mapsto h^{-1}gh$ is a homeomorphism.

Suppose a Lie group $G$ acts smoothly on a manifold $M$. Fix a point $x \in M$ and consider the map $\theta^{(x)}: G \ni g  \mapsto gx \in M$. Passing to the quotient $\Theta^{(x)}: G/G_x \rightarrow M$ by sending the equivalence class $gG_x$ to $gx$ yields an injective smooth immersion \cite[Theorem 7.25]{lee2012smooth} (which is thus diffeomorphic onto its image \cite[Proposition 5.18]{lee2012smooth}; in particular $\Im d(\theta^{(x)})_e = T_xGx$).  Hence the orbit $Gx$ is an immersed submanifold of $M$. Since proper injective smooth immersions are smooth embeddings \cite[Proposition, 4.22]{lee2012smooth}, $Gx$ becomes an embedded submanifold of $M$ when $\Theta^{(x)}$ is proper, and in particular, when $\theta^{(x)}$ is proper (itself true when $\theta$ is proper \cite[Proposition 21.7]{lee2012smooth}). 
For example, consider the natural action on $\R^4$ of the Lie group
$$  G=\left\{\begin{pmatrix}
    R_\theta & 0\\
    0 & R_{a\theta} 
\end{pmatrix} :~\theta \in \R\right\}~~\text{where }R_{\theta}=\begin{pmatrix}
    \cos\theta & -\sin \theta \\
    \sin\theta & \hphantom{ - }  \cos\theta
\end{pmatrix} \text{ and } a\in \R. $$
The orbits of this action generated by points with nonzero entries (at which it is free) are embedded if $a\in\mathbb Q$ and are merely immersed if $a\in \R\setminus\mathbb Q$ \cite[Example 4.20]{lee2012smooth}. 
This is particularly relevant in optimization since the projection onto a $C^k$ embedded submanifold $M$ of $\mathbb{R}^n$ with $k\in \{2,\hdots,\infty,\omega\}$ is single-valued and $C^{k-1}$ on a neighborhood of $M$ ($\omega$ stands for real analytic, $\infty = \infty -1$, and $\omega = \omega -1$) \cite[Theorem 3.2, 3.6, Theorem 3.8, Theorem 4.1]{dudek1994nonlinear}. It thus inherits favorable properties of the projection on closed convex sets. 

\subsubsection{Slices}
\label{subsubsec:Slices}

So far, we have discussed the structure of a single orbit, but we would like to know more. What is the structure of orbits passing near a fixed point? This is crucial for analyzing the subdifferential since it is defined by taking limits of nearby points. To this end, suppose $G$ acts continuously on $M$. Given $A,B \subseteq M$, let 
$$ G(A|B) = \{ g \in G : A \cap gB \neq \emptyset \}$$
following Koszul \cite{koszul1965lectures}.
Note that $G(\{x\}|\{x\}) = G_x$. A slice is a set $A \subseteq M$ such that $G(A|A)A=A$ and the restriction of the action $G\times A \to M$ is open\footnote{A map between topological spaces is open is it maps open sets to open sets.} \cite[Definition p. 12]{koszul1965lectures}.  A slice at a point $x \in M$ is a slice such that $x\in A$ and $G(A|A) = G_x$. A normal slice is a slice $A$ such that $G(A|A) = G_y$ for all $y \in A$. 

If $A$ is a normal slice at $x$, then $G/G_x \times A \to GA$ is a homeomorphism, as explained in the next paragraph. This means that one can view the neighborhood $GA$ of $Gx$ as simply a Cartesian product, with the set $A$ ``slicing'' through $GA$. 

If $ga = hb$ with $g,h\in G$ and $a,b\in A$, then $a = (g\p{-1} h) b$, $g\p{-1} h \in G(A|A) = G_b = G_x$, and $(g\p{-1} h) b = b$. In other words, $g \in h G_x$ and $a=b$. It is thus natural to consider the quotient maps\footnote{Let $X$ be a topological space, $Y$ be any set, and $q:X\to Y$ be a surjective map. The quotient topology on $Y$ is defined by declaring $U$ to be open in $Y$ if and only if $q\p{-1}(U)$ is open in $X$. Given two topological spaces $X$ and $Y$, a quotient map is a surjective map $q:X\to Y$ where $Y$ is endowed with the quotient topology induced by $q$.} $\pi: G\to G/G_x$ and $(q,\mathrm{Id}_A)$. One may pass continuously to the quotient $G/G_x \times A \to M$ \cite[Theorem 3.73]{lee2012smooth} while remaining an open map and becoming injective. Restricting the codomain to the image, i.e., $G/G_x \times A \to GA$, retains continuity and openness, using the subspace topology in the codomain. It thus is a homeomorphism.  

It will be convenient to strengthen Koszul's definitions. Suppose $G$ acts smoothly on $M$. A smooth slice is an embedded submanifold $A \subseteq M$ such that $G(A|A)A=A$ and the restriction of the action $G\times A \to M$ is a smooth submersion (thus an open map \cite[Proposition 4.28]{lee2012smooth}). A smooth slice at a point $x \in M$ is a smooth slice such that $x\in A$ and $G(A|A) = G_x$. A smooth normal slice is a smooth slice $A$ such that $G(A|A) = G_y$ for all $y \in A$.

If $A$ is a smooth normal slice at $x$, then $G/G_x \times A \to GA$ is a diffeomorphism. Indeed, $\pi : G\to G/G_x$ is a surjective smooth submersion by the equivariant rank theorem \cite[Theorem 7.25]{lee2012smooth} and so is $(\pi,\mathrm{Id}_A)$. One may thus pass smoothly to the quotient $G/G_x\times A \to M$ \cite[Theorem 4.30]{lee2012smooth} while remaining a smooth submersion by the chain rule \cite[Proposition 3.6]{lee2012smooth}. The new map is naturally injective, and since it has constant rank, it is a smooth immersion by the global rank theorem \cite[Theorem 4.14]{lee2012smooth}. But we already know it is a homeomorphism onto its image (shown above), so it is a smooth embedding. By \cite[Propositon 5.2]{lee2012smooth}, $G/G_x\times A \to GA$ is a diffeomorphism using the subspace topology in the codomain. In order to understand when a smooth normal slice should exist, we need to introduce two notions: the slice representation and the type of an orbit.

If $G$ acts smoothly on $M$ and $x \in M$, then $G_x$ acts smoothly on $T_xM$. In order to see this, it is convenient to name the action $\theta : G \times M \rightarrow M$ and the action of a single element $\theta_g$. Each element $g\in G_x$ fixes the map $\theta_g:M\rightarrow M$ and taking the derivative at $g$ gives the map $d\theta_g:T_xM\rightarrow T_xM$. By the chain rule, $d(\theta_{gh})_x = d(\theta_g \circ \theta_h)_x = d(\theta_g)_x \circ d(\theta_h)_x$ and thus there is a representation $\rho_x: G_x \rightarrow \mathrm{GL}(T_xM)$ given by $\rho_x(g) = d(\theta_g)_x$. Since $T_xGx$ is stable under $\rho_x(g)$, one actually obtains a slice representation $\sigma_x : G_x \rightarrow \mathrm{GL}(T_xM/T_xGx)$. When $M = \mathbb{R}^n$, one may identify $T_xM/T_xGx$ with the normal space $N_x Gx$. By definition, $G_x$ acts trivially on $T_xM/T_xGx$ if $\sigma_x(g) = e$ for all $g\in G_x$. 

Two subgroups $H,H'$ of a group $G$ are conjugate if $H=gH'g^{-1}$ for some $g\in G$. The set of subgroups of $G$ conjugate to a given subgroup $H$ is called the conjugacy class of $H$ in $G$ \cite[p. 403]{lee2010introduction}. The type of an orbit $Gx$ is the conjugacy class $\tau(x)$ of the isotropy group $G_x$ in $G$. This is motivated by the fact that the isotropy groups of two points on an orbit are conjugate to one another, as evidenced by the relation $G_{gx} = g^{-1}G_xg$ for all $g\in G$. 

%To each isotropy group $H$ one may thus attribute a subset of $\mathbb{R}^n$ defined by the points $x \in \mathbb{R}^n$ for which $G_x$ is conjugate to $H$. These subsets form the strata in the stratification by orbit types \cite[4.3.5]{pflaum2001analytic}, which is actually a Verdier stratification \cite{giacomoni2014stratification}. This stratification is however not suitable for our purposes. Indeed, we are interested in points where the conjugacy class of the isotropy group is locally constant, which would be an interior point of a stratum. We are also interested in tangent spaces to orbits not to strata. This should become clear in what follows.

Suppose $G$ acts smoothly on $M$ and properly near $x\in M$. The slice theorem \cite[Lemma 4]{koszul1965lectures} \cite[4.2.6]{pflaum2001analytic} \cite{koszul1953certains} \cite[Theorem 2.3.2]{palais1961existence} asserts that there exists a smooth slice at $x$ \cite[Theorem 1 p. 17]{koszul1953certains} and 
that the following are equivalent:
\begin{enumerate}[label=\rm{(\roman{*})}]
    \item there exists a smooth normal slice at $x$; \label{item:normal_slice}
    \item $G_x$ acts trivially on $T_xM/T_xGx$; \label{item:trivial}
    \item the orbit type $\tau$ is constant near $x$. \label{item:type}
\end{enumerate}
\ref{item:normal_slice} $\Longrightarrow$ \ref{item:trivial} is due to \cite[Theorem 2 p. 17]{koszul1953certains}. \ref{item:normal_slice} $\Longleftarrow$ \ref{item:trivial} is due to \cite[Proposition 5.2]{lee2012smooth} and \cite[Lemma 3, Remark, Lemma 4 p. 15, Theorem 1, Theorem 2 p. 17]{koszul1953certains}. Finally, \ref{item:trivial} $\Longleftrightarrow$ \ref{item:type} holds by \cite[Lemma 3 p. 15]{koszul1965lectures}. In light of the above equivalences, we propose the following definition.

\begin{definition}
A smooth action of a Lie group $G$ on a smooth manifold $M$ is typical (or acts typically) at $x\in M$ if it is proper near $x$ and the orbit type $\tau$ is constant near $x$.
\end{definition}

%When $G$ acts properly on $M$ near $x$, this is equivalent to the fact that the orbit type $\tau$ is constant in a neighborhood of $x$ in $M$ \cite[Lemma 3 p. 14, Theorem 2 p. 17]{koszul1965lectures}. 

Proper actions, in the global sense, actually induce a global orbit structure, namely, a stratification by orbit types \cite[Theorem 4.3.7]{pflaum2001analytic}. This calls for introducing the subject, following \cite{mather1970notes,trotman2020stratification}.

\subsection{Stratification}
\label{subsec:Stratification}

The distance \cite[(2.1) p. 197]{kato2013perturbation} between two linear subspaces $V,W$ of $\mathbb{R}^n$ is given by 
\begin{equation*}
    d(V,W) = \sup \{ d(v,W) : v \in V, |v|=1\}
\end{equation*}
and $d(V,W) = 0$ if $V = \{0\}$. It satisfies two important properties: 1) $d(V\p \perp, W\p \perp) = d(W,V)$ by \cite[Theorem 2.9]{kato2013perturbation}; 2) if $\mathrm{dim}V = \mathrm{dim}W$, then $d(V,W) = |P_V - P_W|$ by \cite[Lemma 3.2]{morris2010rapidly}. The first is useful when dealing with normal cones. Thes second shows that the distance defines a metric on the Grassmannian $G_k(\mathbb{R}^n)$, namely, the set of $k$-dimensional linear subspaces of $\mathbb{R}^n$. % which by the way implies that $d(V^\perp,W^\perp) = d(V,W)$. 

Let $k$ be a positive integer. A $C^k$ stratification of a subset $S$ of $\mathbb{R}^n$ is a locally finite partition $\mathcal{X}$ of $S$ such that: 
\begin{enumerate}[label=\rm{(\roman{*})}]
    \item Each element $X \in \mathcal{X}$, called stratum, is a $C^k$ embedded submanifold of $\mathbb{R}^n$;
    \item For all $X,Y \in \mathcal{X}$, if $X \cap \overline{Y} \neq \emptyset$ then $X \subseteq \overline{Y}$. 
\end{enumerate}
Since $\mathcal{X}$ is a partition, if $X,Y \in \mathcal{X}$ and $X \neq Y$, then $X\cap Y = \emptyset$. If in addition $X \cap \overline{Y} \neq \emptyset$, then (ii) implies that $X \subseteq \overline{Y} \setminus Y$. A stratification of $\mathbb{R}^{n+1}$ is nonvertical if $(0,\hdots,0,1) \in \mathbb{R}^{n+1}$ is not tangent to any stratum at any point.

A pair of submanifolds $(X,Y)$ of $\mathbb{R}^n$ fulfills the Whitney-(a) condition at $\overline{x} \in X$ if for any $\mathrm{dim}Y$-dimensional linear subspace $\tau \subseteq \mathbb{R}^n$ and any sequence $y_k\in Y \rightarrow \overline{x}$, we have 
$$d(\tau,T_{y_k}Y) \rightarrow 0\quad \implies \quad T_{\overline{x}}X \subseteq \tau.$$ It satisfies the Verdier condition at $\overline{x} \in X$ if 
$$d(T_xX,T_yY) = O(|x-y|)$$ 
for $x\in X$ and $y\in Y$ near $\overline{x}$. Accordingly, a Whitney-(a) (respectively Verdier) stratification is one in which every pair of strata $(X,Y)$ such that $X \subseteq \overline{Y}\setminus Y$ satisfies the Whitney-(a) (respectively Verdier) condition. A prime example of Verdier stratification is given by the determinental variety
$$\R\p{m \times n}_{\leq r} = \{ X \in \R\p{m \times n} : \rank X \leq r\}$$
where the subsets of fixed rank matrices
$$\R\p{m \times n}_{k} = \{ X \in \R\p{m \times n} : \rank X = k\}$$ with $k\in \{0,\hdots,r\}$ form strata \cite[Proposition 7]{ding2014introduction} \cite[Section 4]{hosseini2019gradient}. They are actually orbits of the smooth action of $\GL(m,\R)\times \GL(n,\R)$ on $\R\p{m \times n}$ defined by $(A,B)X = AXB$.

The Whitney condition was introduced to optimization by Bolte \textit{et al.} \cite{bolte2007clarke} to obtain nonsmooth versions of the Morse-Sard theorem and the Kurdyka-\L{}ojasiewicz inequality. When applied to a lower semicontinuous function $f:\mathbb{R}^n\rightarrow \overline{\mathbb{R}}$, it yields the \textit{\textbf{projection formula}} \cite[Proposition 4]{bolte2007clarke}
\begin{equation*}
    P_{T_{\overline{x}}X} \partial f(\overline{x}) \subseteq \{\nabla_X f(\overline{x})\} ~~~ \text{and}~~~ P_{T_{\overline{x}}X} \partial f^\infty(\overline{x})= \{0\}
\end{equation*}
where $X$ is the stratum containing any $\overline{x}\in \mathbb{R}^n$ in the stratification of the domain obtained by projecting a nonvertical stratification of the graph onto it. For any smooth submanifold $X\subseteq \mathbb{R}^n$, the covariant gradient is defined by $\nabla_X f(\overline{x})=P_{T_{\overline{x}}X}\nabla \bar{f}(\overline{x})$ where $\bar{f}$ is any $C^1$ smooth function defined on a neighborhood $U$ of $\overline{x}$ in $\mathbb{R}^n$ and that agrees with $f$ on $U \cap X$. 

The stronger Verdier condition was introduced to optimization by Bianchi \textit{et al.} \cite{bianchi2023stochastic} and Davis \textit{et al.} \cite{davis2025active} to show that a perturbed subgradient method with diminishing step size does not converge to active strict saddle points almost surely. When applied to the graph of a function $f:\mathbb{R}^n\rightarrow \overline{\mathbb{R}}$ that is locally Lipschitz continuous on its domain, it yields the \textit{\textbf{perturbed projection formula}} \cite[Theorem 3.6]{davis2025active}
\begin{gather*}
  \forall v\in\partial f(y),\quad  |\nabla_X f(x)-P_{T_xX}(v)| = O(\sqrt{1+|v|^2}~|x-y|) \\ \text{and} \\  \forall w\in\partial^\infty f(y),\quad |P_{T_xX}(w)| = O(|w||x-y|)
\end{gather*}
for $x \in X$, $y \in \dom \hspace*{.5mm}\partial f$ near $\overline{x}$. This naturally leads to the following definition.

\begin{definition}
\label{defn_vari_strat}
     A $C^k$ variational stratification of a function $f:\R^n\to \overline{\R}$ is a $C^k$ Verdier stratification of $\dom  f$ with finitely many strata such that $f$ is $C^k$ on each stratum and the perturbed projection formula holds at all $\overline{x}\in \dom  f$. 
\end{definition}
As shown by Loi \cite[Theorem 1.3]{le1998verdier}, given a finite family of definable sets, there exists a Verdier stratification of $\mathbb{R}^n$ compatible with each set, meaning that each one is a union of strata. Also, there are finitely many strata and each one is definable. Given a definable function $f:\mathbb{R}^n\rightarrow \overline{\mathbb{R}}$ that is continuous on its domain and a definable set $X \subseteq \dom f$, there thus exists a Verdier stratification of the graph and hence the domain such that $X$ is a finite union of strata \cite[Theorem 3.29]{davis2025active}. Since $X$ may not be contained in any strata, the projection formulae might not hold \cite[Example 2.8]{josz2024sufficient}. Fortunately, by tilting $f$ by a linear function, almost surely around each saddle point $\overline{x}\in \mathbb{R}^n$ there exists a submanifold $X$ containing $\overline{x}$ such that the perturbed projection formula holds \cite[Theorem 5.2]{drusvyatskiy2016generic} \cite[Theorem 2.9]{davis2022proximal} \cite[Theorem 3.31]{davis2025active}. This holds if $f$ is lower semicontinuous and weakly convex (in addition to being definable). It can then be shown that the projection of the iterates on $X$ of a pertubed subgradient method correspond to an inexact Riemannian gradient method with an implicit retraction. 

While this technique enables proving nonconvergence to saddle points, it is not suitable for proving instability of discrete subgradient dynamics around nonstrict local minima. The author and his coauthor \cite[Theorem 2.9]{josz2024sufficient} instead devised a new proof scheme based on the existence of Chetaev function near a nonstrict local minimum $\overline{x}$. For it to work, they assume that the set of local minima near $\overline{x}$ forms a $C^2$ embedded submanifold $X$ and that the perturbed projection formula holds. In \cref{sec:Discrete subgradient dynamics}, we show how these conditions generally hold by symmetry. 

This brings us back to orbits. Recall that the type of an orbit $Gx$ is the conjugacy class $\tau(x)$ of the isotropy group $G_x$ in $G$. To each isotropy group $H$ one may thus attribute a subset of $\mathbb{R}^n$ defined by the points $x \in \mathbb{R}^n$ for which $G_x$ is conjugate to $H$. These subsets form the strata in the stratification by orbit types \cite[4.3.5]{pflaum2001analytic}, which is actually a Verdier stratification \cite{giacomoni2014stratification}. This stratification is however not suitable for our purposes. Indeed, we are interested in points where the conjugacy class of the isotropy group is locally constant, which would be an interior point of a stratum. We are also interested in tangent spaces to orbits not to strata. This should become clear in what follows.

\subsection{Useful facts}
\label{subsec:Useful facts}

We finish the background section with some known facts that will be used later. Their proofs are included for completeness and can be found in the Appendix. The first one provides a criterion for an orbit to be embedded, complementing the standard properness assumption. 

\begin{fact}{\cite[Appendix B]{gibson1979singular}}
    \label{fact:embedded_orbit}
     If the orbit of $C^k$ action on $\R^n$ is definable with $k\in \mathbb N^*$, then it is $C^k$ embedded.
\end{fact}

\cref{fact:embedded_orbit} can be used for example to prove that the set of positive semidefinite matrices with fixed rank is embedded \cite{vandereycken2009embedded}. We next record a standard fact from differential geometry. It will be used for sensitivity analysis.

\begin{fact}
\label{fact:orbit_rank}
    Let $\theta$ be a smooth action of $G$ on a manifold $\cM$ and $\overline{x}\in  \cM$. Then $d(\theta^{(x)})_e$ has constant rank for all $x\in G\overline{x}$.
\end{fact}

The next fact will be used repeatedly to control the distance between tangent spaces of nearby points. While it follows from a more general and involved theory \cite[Section 3]{wedin1972perturbation}, we provide an elementary proof of the special case we are interested in. Let $\mathrm{L}(V,W)$ be the set of linear maps between two finite-dimensional normed vector spaces $V$ and $W$, equipped with the induced norm (again denoted $|\cdot|$). 

\begin{fact}
\label{fact:injective}
    $d(\mathrm{Im}A,\mathrm{Im}B) = O(|A-B|)$ for $A,B \in \mathrm{L}(V,W)$ near an injective map $\bar{A}$.
\end{fact}

\cref{fact:injective} implies that the tangent space of a embedded submanifold of $\mathbb{R}^n$ is locally Lipschitz continuous.  

\begin{fact}{\cite[3.6]{dudek1994nonlinear}}
\label{fact:embedded_tangent}
    If $M$ is a $C\p 2$ embedded submanifold of $\mathbb{R}^n$ and $\overline{x} \in M$, then
    \begin{equation*}
        d(T_xM,T_yM) = O(|x-y|)
    \end{equation*}
    for $x,y\in M$ near $\overline{x}$.
\end{fact}

\cref{fact:embedded_tangent} can be applied to the case where the manifold $M$ is an orbit of a proper action or a definable orbit of a smooth action. However, we actually require a finer result where the points $x$ and $y$ potentially lie on different orbits. That will the object of two upcoming lemmata. The next fact will be used to convert a stratification of the graph of a function to its domain. The case of Whitney stratifications is stated under \cite[Remark 3]{bolte2007clarke}.

\begin{fact}
\label{fact:verdier}
    Suppose $f:\R^n\to\overline{\R}$ is locally Lipschitz on its domain. The projection onto $\dom f$ of a $C^k$ Verdier stratification of $\gph f$ is a $C^k$ Verdier stratification of $\dom  f$ such that $f$ is $C\p k$ on each stratum.
\end{fact}

The final fact will be used to infer the perturbed projection formula from a Verdier stratification of the graph of a function. The result is used in the proof of \cite[Theorem 3.30]{davis2025active}.

\begin{fact}
\label{fact:perturbed_formula}
    Suppose the graph of a lower semicontinuous function $f:\R^n\to\overline{\R}$ admits a Verdier stratification $\cX$. Then for all $X\in \cX$ and $(x,f(x))\in X$, we have $$N_{\epii f}(x,f(x))\subseteq N_{(x,f(x))}X.$$
\end{fact}
 
\section{Orbital projection formulae}
\label{sec:orbit_proj}

In order to establish the desired orbital projection formulae, we begin with two lemmata.

\subsection{Tangent spaces to orbits}
\label{subsec:Tangent spaces to orbits}

We begin with an easy case. 

\begin{lemma}
    \label{lemma:free_tangent}
    If $G$ acts smoothly on $\R^n$ and freely at $\overline x\in\R^n$, then 
    \begin{equation*}
        d(T_xGx,T_yGy) = O(|x-y|)
    \end{equation*} 
    for $x,y$ near $\overline x$.
\end{lemma}
\begin{proof}
   Let $\theta:G\times \R^n\to \R^n$ be the action of $G$. We have $T_xGx=\mathrm{Im}d(\theta^{(x)})_e$ for all $x$ near $\overline{x}$. Let $(U,\varphi)$ be a chart at $e$. Define $\widehat\theta:\vp(U)\times \R^n\to \R^n$ as  $\widehat\theta(v,x)=\theta(\vp^{-1}(v),x)$. By the chain rule \cite[Proposition 3.6(b)]{lee2012smooth}, we have $d(\widehat\theta^{(x)})_{\vp(e)}=d(\theta^{(x)})_e\circ d(\vp^{-1})_{\vp(e)}$. Since $\varphi$ is a diffeomorphism, by \cite[Proposition 3.6(d)]{lee2012smooth} $d\varphi^{-1}$ is an isomorphism at $\vp(e)$. Thus, $\Im d(\widehat\theta^{(x)})_{\vp(e)}=\Im d(\theta^{(x)})_e$. Since $G$ acts freely at $\overline x$, $\theta^{(\overline{x})}$ is injective. Indeed, for all $g,h\in G$, if $\theta^{(\overline{x})}(g)=\theta^{(\overline{x})}(h)$, then $\theta(gh^{-1},\overline{x})=\overline{x}$ and $gh^{-1}\in G_{\overline{x}}=\{e\}$, i.e., $g=h$.  By the equivariant rank theorem \cite[Theorem 7.25]{lee2012smooth}, $\theta^{(\overline x)}$ has constant rank, and is thus a smooth immersion. In particular, $d(\theta^{(\overline x)})_e$ is injective, so is $d(\widehat\theta^{(\overline{x})})_{\vp(e)}$. By \cref{fact:injective}, $$d(T_xGx,T_yGy)=d(\Im d(\widehat\theta^{(x)})_{\vp(e)},\Im d(\widehat\theta^{(y)})_{\vp(e)})\leq  C|d(\widehat\theta^{(x)})_{\vp(e)}-d(\widehat\theta^{(y)})_{\vp(e)}|\leq  CL|x-y|$$ for some $C>0$. The existence of a constant $L>0$ is due to the mean value theorem.
\end{proof}

We next consider a harder case, for which we rely on the slice theorem. 

\begin{lemma}
    \label{lemma:orbit}
    If $G$ acts smoothly on $\mathbb{R}^n$ and typically at $\overline{x} \in \mathbb{R}^n$, then 
    \begin{equation*}
        d(T_xGx,T_yGy) = O(|x-y|)
    \end{equation*} 
    for $x,y\in \mathbb{R}^n$ near $\overline{x}$.
\end{lemma}
\begin{proof}
    Since the action is typical at $\overline{x}$, by the slice theorem there exists a smooth normal slice $A \subseteq \R\p n$ at $\overline{x}$. Hence $GA$ is an open subset of $\R\p n$ and $\Theta : G/G_{\overline{x}} \times A \to GA$ is a diffeomorphism.     
    Let $x,y \in \mathbb{R}^n$ be near $\overline{x} = \Theta(G_{\overline{x}},\overline{x}) \in GA$, and hence in $GA$. Let $(g_x,a_x) = \Theta^{-1}(x)$ and $(g_y,a_y) = \Theta^{-1}(y)$. The inclusion $\iota: G/ G_{\overline{x}} \times \{a_x\} \hookrightarrow G/ G_{\overline{x}} \times A$ is a local diffeomorphism, hence so is the composition $\Theta^{(a_x)} = \Theta \circ \iota$ \cite[Proposition 4.6]{lee2012smooth}.
    Restricting the codomain to its image $Gx$ yields a bijective local diffeomorphism, hence a diffeomorphism from $G/G_{\overline{x}}$ to $Gx$, and in particular an isomorphism from $T_{g_x} (G/G_{\overline{x}})$ to $T_x Gx$. It follows that $Gx$ is an embedded submanifold of $\mathbb{R}^n$, $\mathrm{Im}\hspace*{.5mm}d(\Theta^{(a_x)})_{g_x} = T_xGx$, and $\mathrm{dim}\hspace*{.5mm}T_xGx = \mathrm{dim}\hspace*{.5mm}T_yGy$. Let $(U,\varphi)$ and $(V,\psi)$ be charts of $G/G_{\overline{x}}$ at $G_{\overline x}$ and $A$ at $\overline{x}$. Define $\widehat\Theta:\varphi(U)\times \psi(V)\to \R^n$ as  $\widehat\Theta(u,v)=\Theta(\varphi^{-1}(u),\psi\p{-1}(v))$, and $\widehat \Theta^{(v)}$ accordingly. Since $d(\Theta^{(\overline{x})})_{G_{\overline{x}}}$ is injective, so is $d(\widehat\Theta^{(\psi(\overline{x}))})_{\varphi(G_{\overline{x}})}$. Applying \cref{fact:injective} to $d(\widehat\Theta^{(\psi(\overline{x}))})_{\psi(G_{\overline{x}})}$ yields 
    \begin{align*}
           d(T_xGx,T_yGy) &= d(\mathrm{Im}\hspace*{.5mm}d(\Theta^{(a_x)})_{g_x},\mathrm{Im}\hspace*{.5mm} d(\Theta^{(a_y)})_{g_y}) \\
           &=d(\mathrm{Im}\hspace*{.5mm}d(\widehat\Theta^{(\psi(a_x))})_{\vp(g_x)},\mathrm{Im}\hspace*{.5mm} d(\widehat\Theta^{(\psi(a_y))})_{\vp(g_y)}) \\
           &\leq  C|d(\widehat\Theta^{(\psi(a_x))})_{\vp(g_x)}- d(\widehat\Theta^{(\psi(a_y))})_{\vp(g_y)}| \\
           &\leq  CL|(\vp(g_x),\psi(a_x))-(\vp(g_y),\psi(a_y))|\\
           &=CL|(\vp,\psi)\circ \Theta^{-1}(x)- (\vp,\psi)\circ\Theta^{-1}(y)|\\
           &=CLL'|x-y|
    \end{align*}
    for some $C,L,L'>0$, where we use the mean value theorem twice. 
\end{proof}

\subsection{Formulae}
\label{subsec:Formulae}

We are now ready to obtain the desired projection formulae. The power of \cref{lemma:orbit} can be felt in the following propositions. The first can be viewed as an orbital projection formula. 

\begin{proposition}
    \label{prop:orbital_projection_formula}
    If $f:\mathbb{R}^n\rightarrow \overline{\mathbb{R}}$ is $G$-invariant typically at $\overline{x} \in \dom f$, then
    \begin{equation*}
        P_{T_{\overline{x}}G\overline{x}}\partial f(\overline{x}) \subseteq \{0\} ~~~ \text{and} ~~~ P_{T_{\overline{x}}G\overline{x}} \partial^\infty f(\overline{x}) = \{0\}.
    \end{equation*}
\end{proposition}

\begin{proof}
    First observe that $\widehat{\partial} f(\overline{x}) \subseteq \widehat{N}_{G\overline{x}}(\overline{x})$. Indeed, let $v \in \widehat{\partial} f(\overline{x})$. For $x \in \mathbb{R}^n$ near $\overline{x}$, we have
    \begin{equation*}
        f(x) \geq   f(\overline{x}) + \langle v , x - \overline{x} \rangle + o(|x-\overline{x}|).
    \end{equation*}
    In particular, for $x \in G\overline{x}$ near $\overline{x}$, by invariance $f(x) = f(\overline{x})$ and so $\langle v , x - \overline{x} \rangle \leq  o(|x-\overline{x}|)$. This means that $v\in \widehat{N}_{G\overline{x}}(\overline{x})$. Since $G\overline{x}$ is an embedded submanifold of $\mathbb{R}^n$, we in fact have $\widehat{\partial} f(\overline{x}) \subseteq \widehat{N}_{G\overline{x}}(\overline{x}) = N_{\overline{x}}G\overline{x}$.

    We next show that $\partial f(\overline{x}) \subseteq N_{\overline{x}}G\overline{x}$.
    Let $v \in \partial f(\overline{x})$. There is a sequence $(x_k,v_k) \in \gph \hspace*{.5mm} \widehat{\partial} f$ such that $(x_k,f(x_k),v_k) \rightarrow (\overline{x},f(\overline{x}),v)$. By the previous paragraph, $v_k \in \widehat{\partial} f(x_k) \subseteq N_{x_k} Gx_k$. Lemma \ref{lemma:orbit} ensures that
    \begin{equation*}
        d(\mathrm{span}(v_k),N_{\overline{x}} G\overline{x}) \leq  d(N_{x_k} Gx_k,N_{\overline{x}} G\overline{x}) = d(T_{x_k} Gx_k,T_{\overline{x}} G\overline{x}) = O(|x_k - \overline{x}|).
    \end{equation*}
    %As in the proof of Lemma \ref{lemma:embedded_normal}, we conclude that $v\in N_{\overline{x}} G\overline{x}$.
    Without loss of generality, we can assume that $v\neq 0$ and thus $v_k$ are eventually nonzero. Since $d(\mathrm{span}(v_k),N_{\overline{x}} G\overline{x}) = d(v_k/|v_k|,N_{\overline{x}} G\overline{x})$, there exists a sequence $w_k\in N_{\overline{x}} G\overline{x}$ such that $|v_k/|v_k|-w_k|$ converges to zero. Since $v_k/|v_k|$ converges to $v/|v|$, so does $w_k$. The closed set $N_{\overline{x}} G\overline{x}$ thus contains $v/|v|$ and $v$ of course.

    Finally, we show that $\partial^\infty f(\overline{x}) \subseteq N_{\overline{x}}G\overline{x}$. Let $v \in \partial^\infty f(\overline{x})$. There are sequences $\tau_k\searrow 0$ and $(x_k,v_k) \in \gph \hspace*{.5mm}\widehat{\partial} f$ such that $(x_k,f(x_k),\tau_k v_k) \rightarrow (\overline{x},f(\overline{x}),v)$. Replacing $v_k$ by $\tau_k v_k$ in the previous paragraph, we see that the proof is the same.
\end{proof}
The second can be viewed as a perturbed orbital projection formula.
\begin{proposition}
    \label{prop:perturbed_orbital_projection_formula}
    Let $f:\mathbb{R}^n\rightarrow \overline{\mathbb{R}}$ is $G$-invariant typically at $\overline{x} \in \dom f$, then
    %\begin{empheq}[box=\fbox]{gather*}
    \begin{gather*}
       \forall v\in\partial f(y),\quad  |P_{T_xGx}(v)| = O(|x-y|\hspace*{.3mm}|v|) \\ ~~
        \text{and}~~ \\ \forall w\in\partial^\infty f(y),\quad |P_{T_xGx}(w)| = O(|x-y|\hspace*{.3mm}|w|).
    %\end{empheq}
    \end{gather*}
    for $x \in \dom f$, $y \in \dom \hspace*{.5mm}\partial f$ (respectively $y \in \dom f$) near $\overline{x}$.
\end{proposition}
\begin{proof}
For all $x \in \dom f$, $y \in \dom \hspace*{.5mm}\partial f$ near $\overline{x}$, we have
\begin{subequations}
\begin{align}
\forall y\in\partial f(y),\quad    |P_{T_xGx}(v)| & = |P_{T_xGx}(v)-P_{T_yGy}(v)| \label{sopf_a} \\
    & \leq  |P_{T_xGx}-P_{T_yGy}||v| \label{sopf_b} \\
    & = d(T_xGx,T_yGy)|v| \label{sopf_c} \\
    & = O(|x-y|\hspace*{.3mm}|v|) \label{sopf_d}
\end{align}
\end{subequations}
Indeed, \eqref{sopf_a} is due to Lemma \ref{prop:orbital_projection_formula}. \eqref{sopf_b} uses the definition of the operator norm. \eqref{sopf_c} follows from $\mathrm{dim}T_xGx = \mathrm{dim}T_yGy$ and \cite[Lemma 3.2]{morris2010rapidly}. Finally, \eqref{sopf_d} is due to Lemma \ref{lemma:orbit}. One argues in the same fashion for $\partial^\infty f$.
\end{proof}
When comparing the two above formulae with the projection formulae discussed in \cref{subsec:O-minimal structures}, bear in mind that $\nabla_{Gx}f(x)=0$ since $f$ agrees with a constant function on $Gx$.
Note that Propositions \ref{prop:orbital_projection_formula} and \ref{prop:perturbed_orbital_projection_formula} also hold under the assumptions of \cref{lemma:free_tangent}, which requires a free action. Without the free or typical assumptions, one needs to make an assumption on the objective function $f$, as follows. Below, $\mathrm{int}$ denotes the interior.
\begin{proposition}
    \label{prop:projection_formula_ll}
    If $f:\R^n\to\overline{\R}$ is $G$-invariant and Lipschitz continuous near $\overline{x}\in\rint \dom  f$, then $$P_{T_{\overline{x}}G\overline{x}}  \partial f(\overline x)=\{0\}.$$
\end{proposition}
\begin{proof}
   Let $\theta$ denote the action. Suppose $f$ is differentiable at $x$. Since $\theta^{(x)}$ is smooth, there exists a chart $(U,\varphi)$ around $e\in G$ such that $\theta_{ x}\circ \varphi^{-1}:U\to\R^n$ is smooth. By the invariance of $f$ and the chain rule \cite[Proposition 3.6(b)]{lee2012smooth}, we have $$0=d(f\circ\theta^{(x)}\circ \vp^{-1})_{\vp(e)}=df( x)\circ d(\theta^{(x)}\circ \vp^{-1})_{\vp(e)}=df(x)\circ d(\theta^{(x)})_e \circ d(\vp^{-1})_{\vp(e)}.$$ Since $\vp^{-1}$ is a diffeomorphism, $d(\vp^{-1})_{\vp(e)}$ is an isomorphism \cite[Proposition 3.6(d)]{lee2012smooth}. Thus 
    \begin{equation}
        \label{eq:ll}
        \forall v\in \fg,\quad  
        0=df(x)( d(\theta^{(x)})_e(v))=\langle \nabla f(x),d(\theta^{(x)})_e(v)\rangle.
    \end{equation}
    Let $\widehat\theta:\vp(U)\times \R^n\to \R^n $ be defined by $\widehat\theta(v,x)= \theta(\vp^{-1}(v),x)$, and consider $\widehat\theta^{(x)}:U\ni v\mapsto \widehat\theta(v,x)$. We have $d(\widehat\theta^{(x)})_{\vp(e)}=d(\theta^{(x)})_e\circ d(\vp^{-1})_{\vp(e)}$, and thus $d(\theta^{(x)})_e=d(\widehat\theta^{(x)})_{\vp(e)}\circ (d(\vp^{-1})_{\vp(e)})^{-1}=d(\widehat\theta^{(x)})_{\vp(e)}\circ d\vp_e$. Let $e_i$ denote the canonical basis of $\R^k$, and let $b_i=d(\vp^{-1})_{\vp(e)}(e_i)$. For any $v=v^ib_i\in \fg$, we have $$d(\theta^{(x)})_e(v)=v^i d(\theta^{(x)})_e(b_i)=v^i d(\widehat\theta^{(x)})_{\vp(e)}(d\vp_e (b_i))=v^i d(\widehat\theta^{(x)})_{\vp(e)}(e_i)=v^i\frac{\partial \widehat\theta}{\partial v_i}(\varphi(e),x).$$ We conclude that the function $\R^n\ni x\mapsto d(\theta^{(x)})_e(v) $ is continuous for any $v\in \fg$. Passing to the limit in \eqref{eq:ll} yields $\langle \overline\nabla f(\overline x), d(\theta^{(\overline  x)})_e(v)  \rangle=\{0\}$ for all $ v\in \fg.$ Thus, $\langle \overline\nabla f(\overline x), T_{\overline{x}}G\overline{x} \rangle=\{0\}$. Since $f$ is Lipschitz continuous near $\overline x$, by \cite[Theorem 9.61]{rockafellar2009variational}, we have $\mathrm{co} \overline{\nabla} f(\overline{x}) = \mathrm{co}\hspace*{.3mm} \partial f(\overline{x})$. Hence $\{0\}=\langle \overline \nabla f(\overline{x}),T_{\overline{x}}G\overline{x}\rangle =\langle \mathrm{co}\hspace*{.3mm} \overline \nabla f(\overline{x}),T_{\overline{x}}G\overline{x}\rangle    =\langle \mathrm{co}\hspace*{.3mm}\partial f(\overline x),T_{\overline{x}}G\overline{x}\rangle=\langle \partial f(\overline x), T_{\overline{x}}G\overline{x}\rangle.$ 
\end{proof}
Note that $x$ is required to lie in the orbit in the formula below, in contrast to \cref{prop:perturbed_orbital_projection_formula}.
\begin{proposition}
    \label{prop:perturbed_projection_formula_ll}
    If $f:\R^n\to\overline{\R}$ is $G$-invariant and Lipschitz continuous near $\overline x\in\rint \dom  f$, then $$\forall v\in\partial f(y),\quad |P_{T_xGx}(v) |=O(|x-y|\hspace*{.3mm}|v|)$$ for $x\in G\overline{x}$ and $y\in \R\p n$ near $\overline{x}$.
\end{proposition}
\begin{proof}
    Let $(U,\vp)$ be a chart at $e\in G$. Let $\widehat\theta=\theta\circ (\vp^{-1},\mathrm{Id}_{\R^n})$ and $\widehat\theta^{(x)}=\theta^{(x)}\circ \vp^{-1}$ for all $x\in \R^n$. By the chain rule \cite[Proposition 3.6(b)]{lee2012smooth}, $d(\widehat\theta^{(x)})_{\vp(e)}=d(\theta^{(x)})_e\circ d(\vp^{-1})_{\vp(e)}$. Consider $b_i$ such that $d(\widehat\theta^{(\overline x)})b_i$ is a basis of $\Im d(\widehat\theta^{(\overline x)})_{\vp(e)}$, as well as the linear map $B(y)=y^ib_i$. Define $\Lambda_x=d(\widehat\theta^{(x)})_{\vp(e)}\circ B$ for all $x\in \R^n$. The map $\Lambda_{\overline{x}}$ is injective by construction. Since $\widehat\theta$ is smooth, $\Lambda_x$ is continuous as a function of $x$. Thus, $\dim \Im\Lambda_x\geq  \dim \Im\Lambda_{\overline{x}}$ for all $x$ near $\overline{x}$, and in particular, $\Lambda_x$ is injective. On the one hand, we have $$ \forall x\in \R^n,\quad \Im \Lambda_x \subseteq \Im d(\widehat\theta^{(x)})_{\vp(e)} = \Im d(\theta^{(x)})_e = T_xGx.  $$  
    By \cref{prop:projection_formula_ll}, it follows that $\partial f(y)\subseteq (\fg  y)^\perp\subseteq (\Im \Lambda_y)^\perp$ for $y$ near $\overline{x}$. So $P_{\Im\Lambda_y}\partial f(y)=\{0\}$. On the other hand,  
    $$\forall x\in G\overline{x},\quad \Im \Lambda_x=T_xGx.  $$
    Indeed, by \cref{fact:orbit_rank}. for all $x\in G\overline{x}$, we have $\dim \Im\Lambda_{\overline x}\leq  \dim \Im\Lambda_x \leq  \dim\Im d\widehat\theta^{(x)}_{\vp(e)}=\dim\Im d(\theta^{(x)})_e=\dim\Im d(\widehat\theta^{(\overline x)})_{\vp(e)}=\dim \Im\Lambda_{\overline{x}}$. Thus, $\dim\Im \Lambda_x=\dim \Im d(\widehat\theta^{(x)})_{\vp(e)}$ and $\Im\Lambda_x=\Im d(\widehat\theta^{(x)})_{\vp(e)}=T_xGx$. By \cref{fact:injective}, for all $x\in G\overline{x}$ and $y$ near $\overline{x}$, and all $v\in\partial f(y)$, we have 
\begin{align*}
    |P_{T_xGx}(v)| & = |P_{\Im \Lambda_x}(v)-P_{\Im\Lambda_y}(v)|  \\
    & \leq |P_{\Im \Lambda_x}-P_{\Im \Lambda_y}|  \hspace*{0.5mm}|v|  \\
    & =L\hspace*{0.5mm}d(\Im \Lambda_x,\Im \Lambda_y)\hspace*{0.5mm}|v|  \\
    & \leq   C\hspace*{.5mm}|\Lambda_x-\Lambda_y|\hspace*{0.5mm}|v|   \\
    & \leq  C\hspace*{.5mm}|d(\widehat\theta^{(x)})_{\vp(e)}-d(\widehat\Theta^{(y)})_{\vp(e)}|  \hspace*{0.5mm}|B|\hspace*{0.5mm}|v|  \\
    & \leq  C|B|\hspace*{0.5mm}L\hspace*{0.5mm}|x-y| \hspace*{0.5mm}|v|
\end{align*}
for some constants $C,L>0$. 
\end{proof}
Thanks to \cref{fact:embedded_orbit} and \cref{prop:perturbed_projection_formula_ll}, the embeddedness and Verdier assumptions in the instability result \cite[Theorem 2.9]{josz2024sufficient} of the subgradient method hold so long as the level set of $f$ locally agrees with an orbit of a smooth action. We will get back to this in \cref{sec:Discrete subgradient dynamics}. 
\begin{proposition}
     \label{prop:orbital_projection_formula_boundary}
    If $f:\R^n\to\overline{\R}$ is regular at $\overline{x}\in \R^n$ and constant on an immersed submanifold $M$ of $\R^n$, then \begin{equation*}
        P_{ T_{\overline x}M}\partial f(\overline{x}) \subseteq \{0\} ~~~ \text{and} ~~~ P_{T_{\overline x}M} \partial^\infty f(\overline{x}) = \{0\}.
    \end{equation*}
    In particular, this holds when $f=g+\delta_C$ where $g:\R^n\to \R$ is Lipschitz continuous near $\overline{x}$, and $g$ and $C\subseteq \R^n$ are regular at $\overline{x}$.
\end{proposition}
\begin{proof}
    Let $v\in \widehat\partial f(\overline{x})$ and $\gamma: J\to \R^n$ be a smooth curve such that $\gamma(0)=\overline{x}$ and $\gamma(J)\subseteq M$, where $J$ is an interval of $\R$ containing $0$. We have 
    $$   f(\gamma(t))\geq f(\gamma(0))+\langle v,\gamma(t)-\gamma(0) \rangle+o(|\gamma(t)-\gamma(0)|)                       $$
    for $t\in J$ near $0$. Since $f(\gamma(t))=f(\gamma(0))$, we have $0\geq \langle v,\gamma(t)-\gamma(0) \rangle+o(|\gamma(t)-\gamma(0)|)$. Thus $0\geq \langle v, (\gamma(t)-\gamma(0))/|\gamma(t)-\gamma(0)| \rangle+o(1)$. Passing to the limit yields $0\geq \langle v,\gamma'(0)\rangle$. Since $\gamma'(0)\in T_{\overline{x}}M$ by \cite[Proposition 5.35]{lee2012smooth} and $T_{\overline{x}}M$ is a vector space, we have $0\geq \langle v,-\gamma'(0) \rangle$. Hence $\langle v,\gamma'(0) \rangle=0$ and $v\in T_{\overline{x}}M^\perp$. Thus 
    $$     P_{ T_{\overline x}M}\widehat\partial f(\overline{x}) \subseteq \{0\}.    $$
    Since $f$ is regular at $\overline{x}$, by \cite[Corollary 8.11]{rockafellar2009variational} we have 
    \[    \partial f(\overline{x})=\widehat\partial f(\overline x),\quad \partial^\infty f(\overline{x})=\widehat\partial f(\overline{x})^\infty.          \]
    By \cite[Theorem 8.6]{rockafellar2009variational} $\widehat\partial f(\overline{x})$ is closed and convex. Together with \cite[Theorem 3.6]{rockafellar2009variational} we find that $ P_{T_{\overline x}M} \partial^\infty f(\overline{x}) = \{0\}.$ The conclusion now readily follows. As for the particular case,  since $g$ is Lipschitz continuous near $\overline{x}$, by \cite[Theorem 9.13]{rockafellar2009variational} we have $\partial^\infty g(\overline{x}) = \{0\}$. Hence the only combination of vectors $(v,w) \in \partial^\infty g(\overline{x}) \times \partial^\infty \delta_C(\overline{x})$ such that $v+w=0$ is $v=w=0$. Since $\delta_C$ is regular at $\overline{x}$, by \cite[Corollary 10.9]{rockafellar2009variational} $f$ is regular at $\overline{x}$.
\end{proof}  
    
We finally consider the case of conservative fields.

\begin{proposition}
    \label{prop:projection_formula_CF}
    Let $D:\R^n\rightrightarrows\R^n$ be a definable conservative field with $G$-invariant definable potential $f:\R^n\to\overline{\R}$. Suppose $d(0,D)$ is locally bounded and $G\overline{x}$ is definable for some $\overline{x}\in \dom f$. Then 
    $$\forae x\in G\overline{x},~P_{T_xGx}D(x)=\{0\}.$$ 
\end{proposition}
\begin{proof}
  By \cref{fact:embedded_orbit}, $G\overline{x}$ is embedded, thus every point in $G\overline{x}$ has a definable neighborhood in $G\overline{x}$. Let $x\in G\overline{x}$. Since $d(0,D)$ is locally bounded and $D$ has closed graph, there exists a definable open neighborhood $U$ of $x$ in $G\overline{x}$ and $r>0$ such that $D\cap B_r(0)$ is nonempty compact valued on $U\cap\dom  f$. By definition, we know $P_{TG\overline{x}}(D\cap B_r(0))$ is a conservative field on $U$ with potential $f+\delta_U$. By \cite[Theorem 4]{bolte2020conservative}, there exists stratification $\cX$ of $U$ such that $ P_{TX}(D\cap B_r(0)) =\{0\}$ for all $X\in \cX$. When $X$ has maximal dimension, $T_{x'}X=T_{x'}G\overline{x}$. Thus $ P_{TG\overline{x}}(D\cap B_r(0))=\{0\} $ almost surely on $U$. Since this holds for all $r$ large enough, we have $$P_{TG\overline{x}} D=P_{TG\overline{x}} \bigcup_{r\in \mathbb{N}}D\cap B_r(0)=\bigcup_{r\in \mathbb N} P_{TG\overline{x}}(D\cap B_r(0))\overset{\text{a.e.}}{=}\{0\}\text{ on }U. $$ This equation holds almost everywhere on $G\overline{x}$ because $G\overline{x}$ is second countable.
\end{proof} 
\section{Invariant variational stratification}
\label{sec:invariant}
The previous section demonstrated how symmetries yield projection formulae similar to those of definable functions. In this section, we show that if we combine both ingredients, namely symmetry and definability, then we can obtain a $G$-invariant variational stratification, meaning that each stratum is a $G$-space. 

%The previously used slice theorem is proved using the tubular neighborhood theorem. It thus relies on the exponential mapping and hence on flows of smooth vector fields. 
The slice theorem is proved by applying the tubular neighborhood theorem on an orbit space. Even if the action is definable, it is thus not clear if the slice is definable.
% if action is definable and typical, maybe one conclude with definable tubular neighborhood theorem? \cite[Theorem 6.11]{coste2000introduction}
Fortunately, if the action is free and proper, a definable slice does exist, as we next show. For notational convenience, if $G$ acts smoothly on $\R\p n$, then let $\vec{N}_xGx = \{x\}+N_xGx$ denote the shifted normal space at $x\in \R\p n$.

\begin{lemma}
\label{lemma:normal_space_slice}
   If $G$ acts smoothly on $\R^n$, freely at $\overline{x}$, and properly near $\overline{x}$, then there exists a definable smooth normal slice $U\subseteq \vec{N}_{\overline{x}}G\overline{x}$ at $\overline{x}$.
\end{lemma}
\begin{proof} 
We first show that there exists a neighborhood $U$ of $\overline{x}$ in $\vec{N}_{\overline{x}}G\overline{x}$ such that $G(U|U) = \{e\}$. We argue by contradiction and assume that there exist $g_k \in G \setminus \{e\}$ and $\vec{N}_{\overline{x}}G\overline{x} \ni x_k,y_k \to \overline{x}$ such that $g_k x_k = y_k$. Since the action is proper near $\overline{x}$, there exists a neighborhood $V$ of $\overline{x}$ in $\R\p n$ such that $G(V|V)$ has compact closure. As $x_k,y_k$ eventually lie in $V$, $g_k$ belongs to a compact subset of $G$, which is sequentially compact since $G$ is second countable. Therefore $g_k\to g\in G$, after taking a subsequence if necessary. Passing to the limit in $g_kx_k=y_k$ yields $g  \overline{x}=\overline{x}$, i.e., $g \in G_{\overline{x}}=\{e\}$ since the action is free at $\overline{x}$. 

Consider the restriction $\Theta: G\times \vec{N}_{\overline{x}}G\overline{x} \to \R\p n$ of the action $\theta$ to $G\times \vec{N}_{\overline{x}}G\overline{x}$. Let $(W,\vp)$ be a chart of $e$ in $G$, and $\widehat\Theta:\vp(W)\times \vec{N}_{\overline{x}}G\overline{x} \to \R^n$ be defined by $\widehat\Theta(z,x)=\Theta(\vp^{-1}(z),x)$. Observe that $\widehat\Theta(\vp(g_k),x_k)-\widehat\Theta(\vp(e),x_k)=y_k-x_k$, $\vp(g_k)\neq \vp(e)$, and $\vp(g_k)\to \vp(e)$. By choosing a subsequence if necessary, we have 
    $$\lim_{k\to\infty} \frac{\vp(g_k)-\vp(e)}{|\vp(g_k)-\vp(e)|}\to u$$ for some $|u|=1$. 
    Since $\theta$ is smooth and restricting the domain of a smooth map to an immersed submanifold retains smoothness \cite[Theorem 5.27]{lee2012smooth}, $\Theta$ and $\widehat \Theta$ are smooth. It follows that $$d(\widehat\Theta^{(\overline{x})})_{\vp(e)}u = \lim_{k\to\infty}\frac{\widehat\Theta(\vp(g_k),x_k)-\widehat\Theta(\vp(e),x_k)}{|\vp(g_k)-\vp(e)|} = \lim_{k\to\infty}\frac{y_k-x_k}{|\vp(g_k)-\vp(e)|}\in N_{\overline{x}}G\overline{x}.$$ As argued in the proof of \cref{lemma:free_tangent}, we have $\Im d(\widehat\Theta^{(\overline{x})})_{\vp(e)}=\Im d(\theta^{(\overline x)}
    )_e=T_{\overline{x}}G_{\overline{x}}$, which implies that $d(\widehat\Theta^{(\overline{x})})_{\vp(e)}u\in T_{\overline{x}}G\overline{x}\cap N_{\overline{x}}G\overline{x} =\{0\}$. However, this contradicts the fact that $d(\widehat\Theta^{(\overline{x})})_{\vp(e)}$ is injective as argued in the proof of \cref{lemma:free_tangent}.

    Accordingly, let $U$ be a neighborhood of $\overline{x}$ in $\vec{N}_{\overline{x}}G\overline{x}$ such that $G(U|U) = \{e\}$. After possibly reducing $U$, we can assume that $U$ is definable. Clearly, $U$ is an embedded submanifold of $\R\p n$ such that $\overline{x}\in U$, $G(U|U)U = U$, and $G(U|U) = G_x$ for all $x\in U$. All that remains to show is that $\theta|_{G \times U}$ is a smooth submersion. Since $U$ is open in $\vec{N}_{\overline{x}}G\overline{x}$, this is equivalent to showing that $d\Theta_{(g,x)}$ is surjective for any $(g,x) \in G \times U$. For all $x \in \vec{N}_{\overline{x}}G\overline{x}$, observe that
    \begin{align*}
        \Im d\Theta_{(e,x)} & = \Im d\widehat\Theta_{(\vp(e),x)} \\
        & =\Im d(\widehat\Theta^{(x)})_{\vp(e)} +\Im d(\widehat\Theta_{\varphi(e)})_x \\
        & =\Im d(\Theta^{(x)})_{e} +\Im d(\Theta_{e})_{x} \\
        & =\Im d(\theta^{(x)})_e+N_{\overline{x}}G\overline{x} \\
        & =T_{x}Gx+N_{\overline{x}}G\overline{x}.    
    \end{align*}
    In particular, $\Im d\Theta_{(e,\overline{x})} = T_{\overline{x}}G\overline{x}+N_{\overline{x}}G\overline{x} = \R\p n$, so that $d\Theta_{(e,\overline{x})}$ is surjective. Hence so is $d\Theta_{(e,x)}$ for all $x\in U$, after possibly reducing $U$.  Now let $(g,x)\in G\times U$. For all $h\in G$, we have $\theta(g,\theta(h,x)) = \theta(gh,x)$. By differentiating on both sides and using the chain rule \cite[Proposition 3.6(b)]{lee2012smooth}, we see that $d(\theta_g)_x \circ d(\theta^{(x)})_e = d(\theta^{(x)}\circ L_g)_e = d(\theta^{(x)})_g\circ d(L_g)_e$ where $L_g$ is the left action by $g$. Since $L_g$ is a diffeomorphism on $G$, $d(\theta_g)_x \Im d(\theta^{(x)})_e = \Im d(\theta^{(x)})_g$. As a result,
    \begin{align*}
        \Im d\Theta_{(g,x)} & =\Im d(\Theta^{(x)})_g+\Im d(\Theta_{g})_x \\
        & = \Im d(\theta^{(x)})_g+  d(\theta_{g})_xN_{\overline{x}}G\overline{x} \\ & = d(\theta_g)_x(\Im d(\theta^{(x)})_e+N_{\overline{x}}G\overline{x}) \\
        & = d(\theta_g)_x(T_xGx+N_{\overline{x}}G\overline{x}) \\
        & = d(\theta_g)_x \Im d\Theta_{(e,x)}.    
    \end{align*}
    As $\theta_g$ is a diffeomorphism on $\R^n$ and $\Im d\Theta_{(e,x)}=\R^n$, we find that $d\Theta_{(g,x)}$ is surjective.
\end{proof}

We are now ready to give a stratification result. Note that the action is not assumed to be definable.

\begin{theorem} 
\label{thm:stratification}
Let $k\in \mathbb{N}^*$ and $f:\R^n\to \overline{\R}$ be $G$-invariant, locally Lipschitz on its domain, and definable. Suppose $G$ acts freely at $\overline{x}$ and properly near $\overline{x}$. There exists a neighborhood $V$ of $G\overline{x}$ and a $G$-invariant variational stratification of $f|_{V}$ containing $G\overline{x}$. 
\end{theorem}
\begin{proof}
    By \cref{lemma:normal_space_slice}, there exists a definable smooth normal slice $U$ at $\overline{x}$. Hence $\Theta:G\times U \to GU$ is a diffeomorphism ($G/G_{\overline{x}} = G$ since the action is free at $\overline{x}$). Since the restriction $f|_{U}$ is definable, by \cite[Theorem 1.3]{le1998verdier} there exists a $C^k$ Verdier stratification $\cX$ of $\gph f|_{U}$ with finitely many strata. Without loss of generality, $(\overline{x},f(\overline{x}))$ is a stratum of $\cX$. Observe that $G\times \cX$ is a Verdier stratification of $G\times \gph |_U \subseteq \mathrm{GL}(n,\R) \times \R\p{n+1}$. Also $(\Theta,\mathrm{Id}_{\R})$ is a diffeomorphism and $(\Theta,\mathrm{Id}_{\R})(G\times \gph f|_{U}) = \gph f|_{GU}$ since $f$ is $G$-invariant. As Verdier stratifications are invariant under $C\p 2$ diffeomorphisms \cite[Section 1]{le1998verdier},
    $(\Theta,\mathrm{Id}_{\R})(G\times \cX)$ is a Verdier stratification of $\gph f|_{GU}$. By \cref{fact:verdier} and local Lipschitz continuity of $f$, its projection on $\dom f$ is a Verdier stratification of $\dom f$ whose strata are $G$-spaces and include $G\overline{x}$. The function $f$ is also $C\p k$ on each of them. The perturbed projection formula follows from \cref{fact:perturbed_formula} and \cite[Theorem 3.6]{davis2025active}.
    \end{proof}

If one were to apply the stratification by orbit types to $\gph f|_{V}$, a single stratum would be obtained since all orbits are of the same type. This would be vacuous.

\section{Conservation law}
\label{sec:Conservation law}

The orbital projection formulae enable one to derive conservation laws for subgradient dynamics and conservative field dynamics.

\subsection{Subgradient dynamics}
\label{subsec:Subgradient dynamics}

We first generalize the conservation law of gradient dynamics \cite[Proposition 5.1]{zhao2023symmetries} from smooth to nonsmooth objective functions. The expression of the conserved quantity we obtain is new even in the smooth case. Let $I_n$ denote the identity matrix of order $n$. Given a Lie subgroup $G$ of $\mathrm{GL}(n,\mathbb{R})$ equipped with the Frobenius inner product (again denoted $\langle \cdot,\cdot \rangle$), let $\mathfrak{g} = T_{I_n} G$ be the Lie algebra and $\mathrm{s}(\mathfrak{g})$ be its symmetric elements. In contrast to \cite[Proposition 5.1]{zhao2023symmetries}, we do not require the Lie group to be adjoint stable.

% It will be helpful to determine the update rule of the conserved quantity.

% \begin{proposition}
%     \label{prop:update}
%     Let $f:\mathbb{R}^n\rightarrow \mathbb{R}$ be locally Lipschitz and invariant under the natural action of a Lie subgroup $G$ of $\mathrm{GL}(n,\mathbb{R})$, and let $C(x) = P_{\mathrm{s}(\mathfrak{g})}(xx^T)$. If $x \in \mathbb{R}^n$ and $\alpha\in \mathbb{R}$, then 
%     \begin{equation*}
%     \forall v\in\partial f(x),\quad     {C(x - \alpha v) = C(x) + \alpha^2 C(v)}.
%     \end{equation*}
% \end{proposition}
% \begin{proof}
% Since $\langle \partial f(x),\fg x   \rangle=\{0\}$ for all $v\in \partial f(x)$ by the orbital projection formula in \cref{prop:projection_formula_ll}, one has $P_{s(\fg)}(xv^T)=P_{s(\fg)}(vx^T)=0$. Hence $C(x-\alpha v)=P_{s(\fg)}[(x-\alpha v)(x-\alpha v)^T] 
%   =P_{s(\fg)} (xx^T) -\alpha  P_{s(\fg)} (xv^T) -\alpha  P_{s(\fg)} (vx^T)+\alpha^2 P_{s(\fg)} (vv^T) 
%       =P_{s(\fg)}  (xx^T)+\alpha^2 P_{s(\fg)} (vv^T)=C(x)+\alpha^2C(v)$.
% \end{proof}

\begin{corollary}
\label{cor:conserved}
    Let $f:\mathbb{R}^n\rightarrow \overline{\mathbb{R}}$ be invariant under the natural action of a Lie subgroup $G$ of $\mathrm{GL}(n,\mathbb{R})$, and let $C(x) = P_{\mathrm{s}(\mathfrak{g})}(xx^T)$. Fix $\overline{x}\in \dom  f$. Suppose either of the following holds:
    \begin{itemize}
        \item[\rom{1}] the action is free or typical at $\overline x$;
        \item[\rom{2}] $f$ is Lipschitz continuous near $\overline{x}\in\rint \dom  f$ or regular at $\overline{x}$.
    \end{itemize}
    Then, for all $x\in \dom \hspace{.3mm} \partial f$ near $\overline{x}$,
    $$\forall v\in\overline{\partial} f(x),~~\forall \alpha \in \R,~     C(x + \alpha v) = C(x) + \alpha^2 C(v) ~~~\text{and} ~~~ dC(x)(v) = 0.$$
    In particular, $C$
    % Then 
    % \begin{equation*}
    %    % \boxed{
    %    C(x) = P_{\mathrm{s}(\mathfrak{g})}(xx^T)%}
    % \end{equation*}
    is conserved near $\overline{x}$ along the differential inclusion
    \begin{equation*}
      \forae t>0\quad  x'(t) \in - \overline{\partial} f(x(t)).
    \end{equation*}
\end{corollary}
\begin{proof}
By \cref{fact:embedded_orbit} and \cref{prop:orbital_projection_formula}-\ref{prop:projection_formula_ll}, we have $\langle \partial f(x) , \fg x \rangle = \{0\}$ and $\langle \partial^\infty f(x) , \fg x \rangle = \{0\}$ for all $x\in \dom \hspace{.3mm} \partial f$ near $\overline{x}$, whence $\langle \overline{\partial} f(x) ,\fg x \rangle = \langle \overline{\mathrm{co}}\hspace*{.5mm}[\partial f(x)+\partial^\infty f(x)] ,  \fg  x \rangle = \{0\}$. In other words, for all $v \in \overline{\partial} f(x)$ and $w\in \fg$, $\langle v,wx\rangle = \langle vx\p T,w\rangle = \langle xv\p T,w\p T\rangle = 0$, and so $P_{s(\fg)}(vx^T)=P_{s(\fg)}(xv^T)=0$. Hence 
\begin{align*}
    C(x+\alpha v) & = P_{s(\fg)}[(x+\alpha v)(x+\alpha v)^T] 
  =P_{s(\fg)} (xx^T) +\alpha  P_{s(\fg)} (xv^T) +\alpha  P_{s(\fg)} (vx^T)+\alpha^2 P_{s(\fg)} (vv^T) \\ 
      & = P_{s(\fg)}  (xx^T)+\alpha^2 P_{s(\fg)} (vv^T) \\
      & =C(x)+\alpha^2C(v) \\
      & = C(x) + dC(x)(v)\alpha + o(\alpha)
\end{align*}
and $dC(x)(v) = 0$.
%Fix $v \in \mathrm{s}(\mathfrak{g})$ and define the function $q(x) = \langle x ,vx \rangle$. 
For any solution $x(\cdot)$ to the differential inclusion, we have $(C \circ x)'(t) = dC(x(t))(x'(t))=0$ for almost every $t>0$ when $x(t)$ is near $\overline{x}$, and thus $C \circ x$ is constant.
% Hence $\langle x, v x\rangle = \langle xx^T, v \rangle = \langle P_{\mathrm{s}(\mathfrak{g})}(xx^T), v \rangle$ is conserved for all $h \in \mathrm{s}(\mathfrak{g})$. Letting $e_i$ denote an orthonormal basis of $\mathrm{s}(\mathfrak{g})$, we see that $P_{\mathrm{s}(\mathfrak{g})}(xx^T) = \sum_i \langle P_{\mathrm{s}(\mathfrak{g})}(xx^T), e_i \rangle e_i$ is conserved.
\end{proof}
Note that if $G$ is adjoint stable, then $C(x)=P_{\fg}(xx^T)$. Let us illustrate \cref{cor:conserved} with some examples. We begin with a toy example.
\begin{example}[Lorentz group] 
\label{eg:lorentz}
Let $f:\R^n\to\R$ be defined by $$f(x)=(x_1^2+x_2^2+\dots+x_{n-1}^2-x_n^2-1)^2.$$ In other words, $f(x)=(\langle x, Dx \rangle-1)^2$, where $D=\diag(1,\dots,1,-1)$. Observe that $$f(Ax)=(\langle Ax, DAx\rangle-1)^2=(\langle x, A^T DAx\rangle-1)^2=(\langle x, Dx\rangle-1)^2=f(x)
$$ for all $A\in \GL(n,\R)$ such that $A^TDA=D$. The function $f$ is thus invariant under the action of the Lorentz group $$G=\{A\in \GL(n,\R):~A^TDA=D\} $$ whose Lie algebra is given by 
$$\fg=\{B\in \R^{n\times n}:~B^TD+DB=0\}=\left\{\begin{pmatrix}
    A & b \\
    b^T & 0
\end{pmatrix}:~~A^T=-A\in \R^{(n-1)\times (n-1)},~b\in \R^{n-1}\right\}.
$$ 
% This holds because $D$ is invertible, following classical arguments % (see \cite[Example 7.27]{lee2012smooth}).
By \cref{cor:conserved}, a conserved quantity is given by $$C(x)=P_{s(\fg)}(xx^T)=\begin{pmatrix}
0 & \dots &  0 & x_1x_n \\
\vdots & \ddots & \vdots & \vdots \\
0 & \dots & 0 & x_{n-1}x_n \\
 x_1x_n & \dots &  x_{n-1}x_n   &  0
\end{pmatrix} .$$
\end{example}
\begin{example}[$\ell_1$-matrix factorization]
    \label{eg:mf1_conserved}
    Given a matrix $M \in \mathbb{R}^{m\times n}$, consider the locally Lipschitz function
    \begin{equation*}
    \begin{array}{cccc}
         f: & \mathbb{R}^{m\times r}\times \mathbb{R}^{r\times n} & \longrightarrow & \mathbb{R} \\
        & (X,Y) & \longmapsto & \|XY-M\|_1
    \end{array}
    \end{equation*}
    %where $\|A\|_1 = \sum_{i,j} |A_{ij}|$. 
    It is invariant under the action of $\mathrm{GL}(r,\mathbb{R})$ on $\mathbb{R}^{m\times r} \times \mathbb{R}^{r\times n}$ defined by
    \begin{equation*}
        A (X,Y) = (XA,A^{-1}Y) = \begin{pmatrix}
            A^T \otimes I_m & 0 \\
            0 & I_n \otimes A^{-1} 
        \end{pmatrix}
        \begin{pmatrix}
            x \\
            y
        \end{pmatrix}
    \end{equation*}
    where $\otimes$ is the Kronecker product, $x = \mathrm{vec}(X)$, and $y = \mathrm{vec}(Y)$. Thus let
    \begin{equation*}
        G = \left\{\begin{pmatrix}
            A^T \otimes I_m & 0 \\
            0 & I_n \otimes A^{-1} 
        \end{pmatrix} : A \in \mathrm{GL}(r,\mathbb{R}) \right\} \subseteq \mathrm{GL}((m+n)r,\mathbb{R})
    \end{equation*}
    and compute
    \begin{equation*}
        \mathfrak{g} = \left\{\begin{pmatrix}
            B^T \otimes I_m & 0 \\
            0 & -I_n \otimes B
        \end{pmatrix} : B \in \R^{r\times r} \right\}.
    \end{equation*}
    If $e_i$ is the canonical basis of $\mathbb{R}^r$, then taking $B = (e_ie_j^T+e_je_i^T)/(2m^2n^2)$, $i \leq  j$, yields an orthonormal basis of $\mathrm{s}(\mathfrak{g})$. Thus
    \begin{gather*}
        P_{\mathrm{s}(\mathfrak{g})} \left[ \begin{pmatrix}x \\y \end{pmatrix} \begin{pmatrix}x \\y \end{pmatrix}^T\right]\\
        = \\
        \sum_{i\leq  j} \left\langle \begin{pmatrix}x \\y \end{pmatrix} \begin{pmatrix}x \\y \end{pmatrix}^T , \begin{pmatrix}
            \frac{e_ie_j^T+e_je_i^T}{2m^2n^2} \otimes I_m & 0 \\
            0 & -I_n \otimes \frac{e_ie_j^T+e_je_i^T}{2m^2n^2}
        \end{pmatrix} \right\rangle 
        \begin{pmatrix}
            \frac{e_ie_j^T+e_je_i^T}{2m^2n^2} \otimes I_m & 0 \\
            0 & -I_n \otimes \frac{e_ie_j^T+e_je_i^T}{2m^2n^2}
        \end{pmatrix} \\
        = \\
        \sum_{i\leq  j} \left\langle \begin{pmatrix}x \\y \end{pmatrix} , \begin{pmatrix}
             \frac{e_ie_j^T+e_je_i^T}{2m^2n^2} \otimes I_m & 0 \\
             0 & -I_n \otimes \frac{e_ie_j^T+e_je_i^T}{2m^2n^2}
         \end{pmatrix} \begin{pmatrix}x \\y \end{pmatrix} \right\rangle 
         \begin{pmatrix}
             \frac{e_ie_j^T+e_je_i^T}{2m^2n^2} \otimes I_m & 0 \\
             0 & -I_n \otimes \frac{e_ie_j^T+e_je_i^T}{2m^2n^2}
         \end{pmatrix} 
         \\ = \\
         \sum_{i\leq  j} \langle (X,Y) , (X(e_ie_j^T+e_je_i^T), -(e_ie_j^T+e_je_i^T)Y) \rangle
         \begin{pmatrix}
             \frac{e_ie_j^T+e_je_i^T}{4m^4n^4} \otimes I_m & 0 \\
             0 & -I_n \otimes \frac{e_ie_j^T+e_je_i^T}{4m^4n^4}
         \end{pmatrix} 
         \\ = \\
        \sum_{i\leq  j} \langle X^TX-YY^T , e_ie_j^T+e_je_i^T \rangle
        \begin{pmatrix}
            \frac{e_ie_j^T+e_je_i^T}{4m^4n^4} \otimes I_m & 0 \\
            0 & -I_n \otimes \frac{e_ie_j^T+e_je_i^T}{4m^4n^4}
        \end{pmatrix}
    \end{gather*}
    is conserved by \cref{cor:conserved}. This recovers the known result that $X^TX-YY^T$ is conserved \cite[Proposition 4.5]{josz2023certifying} (see \cite{arora2018optimization,du2018algorithmic,josz2023global,marcotte2024abide} in the smooth case). 
    
    The action is proper near and free at $(\overline{X},\overline{Y})$ if $\mathrm{rank}\overline{X} = r$ or $\mathrm{rank}\overline{Y} = r$. Indeed, properness follows from the inequality $\|XA\|_F\geq  \sigma_{\min}(X)\|A\|_F$, where $\sigma_{\min}(A)$ is the minimal singular value of $A$. The action is free at $(\overline X,\overline Y)$. To see why, notice that if $\rank(\overline{X})=r$ and $\overline{X}A=\overline{X}$ where $A\in \GL(r,\R)$, then $\overline{X}^T\overline{X}A=\overline{X}^T\overline{X}$ and $A=(\overline{X}^T\overline{X})^{-1}=\overline{X}^T\overline{X}=I_r$. By \cref{thm:stratification}, $f$ admits an invariant variational stratification near $(\overline{X},\overline{Y})$.
\end{example}

\begin{remark}
\label{rem:conserved}
    If $f$ is invariant under a smooth linear action $\theta:G\times V \rightarrow V$ that is not necessarily in natural form, then a conserved quantity can be expressed as $$C(x) = d(\theta^{(x)})_e^*x$$ under some mild conditions, where $*$ denotes the adjoint. Namely, this holds provided that $G$ is a Lie subgroup of $\mathrm{GL}(W)$ for some vector space $W$, $\mathrm{GL}(V)$ and $\mathrm{GL}(W)$ are equipped with inner products, $g^*\in G$, and $\theta_g ^* = \theta_{g^*}$ for all $g\in G$. %We call this a \textsuperscript{$\ast$}linear action. 
    Indeed, mimicking the proof of \cref{cor:conserved}, one can see that $\langle \overline{\partial} f(x) , d(\theta^{(x)})_e w \rangle = \{0\}$ for all $x\in \dom \hspace{.3mm} \partial f$ near $\overline{x}$ and $w\in \mathfrak{g}$. Also, for all $g\in G$ and $x,v\in V$, we have
    $$
        \langle v, \theta^{(x)}(g) \rangle = \langle v, \theta_g(x)\rangle = \langle \theta^*_g(v), x\rangle = \langle \theta_{g^*}(v), x\rangle = \langle \theta^{(v)}(g^*), x\rangle.
    $$
    Thus $\langle v, d(\theta^{(x)})_e w \rangle = \langle d(\theta^{(v)})_ew^*, x\rangle$ for all $w\in \mathfrak{g}$, and both terms are equal to $0$ if $v \in \overline{\partial} f(x)$. The equality can be rewritten as $\langle d(\theta^{(x)})_e\p *v, w \rangle = \langle w^*, d(\theta^{(v)})_e\p * x\rangle$. 
    Since $G$ is adjoint stable, if $w\in G$, then $w\p *\in G$ and so $\langle d(\theta^{(x)})_e\p *v, w\p * \rangle = \langle (w^*)\p *, d(\theta^{(v)})_e\p * x\rangle = \langle w , d(\theta^{(v)})_e\p * x\rangle$. Thus $d(\theta^{(x)})_e\p * v = d(\theta^{(v)})_e\p * x = 0$ for all $v \in \overline{\partial} f(x)$. It follows that, for all $v \in \overline{\partial} f(x)$ and $\alpha \in \R$,  
    \begin{align*}
        C(x+\alpha v) & = d(\theta^{(x+\alpha v)})_e^*(x+\alpha v) \\
        & = [d(\theta^{(x)})_e+\alpha d(\theta^{(v)})_e]^*(x+\alpha v) \\
        & = d(\theta^{(x)})_e^*x+\alpha d(\theta^{(x)})_e^*v + \alpha d(\theta^{(v)})_e^*x+\alpha\p 2 d(\theta^{(v)})_e^*v \\
        & = C(x)+\alpha\p 2 C(v).
    \end{align*}
    Indeed, by linearity $\theta^{(x+\alpha v)}(g) = \theta^{(x)}(g)+\alpha\theta^{(v)}(g)$ for all $g\in G$, and so $d(\theta^{(x+\alpha v)})_e = d(\theta^{(x)})_e+\alpha d(\theta^{(v)})_e$.
One last note: specifying $\langle d(\theta^{(x)})_e\p *v, w \rangle = \langle w^*, d(\theta^{(v)})_e\p * x\rangle$ to $x = v$ yields 
$\langle d(\theta^{(x)})_e^*x - (d(\theta^{(x)})_e^*x)^* , w \rangle = 0$. Since $d(\theta^{(x)})_e^*x \in \mathfrak{g}$ and $G$ is adjoint stable, $(d(\theta^{(x)})_e^*x)^* \in \mathfrak{g}$. Plugging in $w = d(\theta^{(x)})_e^*x - (d(\theta^{(x)})_e^*x)^*$ yields $d(\theta^{(x)})_e^*x \in \mathrm{s}( \mathfrak{g})$. The advantage of this expression in practice is that it avoids vectorizing and having to look for an orthonormal basis.    
\end{remark}

\begin{example}
\label{eg:mf1_non_canonical}
    Let $\theta(A,X,Y) = A (X,Y)$ denote the action from the previous example. We first check that for all $A \in \mathrm{GL}(r,\mathbb{R})$ and $(H,K) \in \mathbb{R}^{m\times r} \times \mathbb{R}^{r\times n}$, we have
    \begin{align*}
        \langle \theta_A(X,Y) , (H,K) \rangle & = \langle (XA,A^{-1}Y) , (H,K) \rangle \\
        & = \langle XA, H \rangle + \langle A^{-1}Y , K \rangle \\
        & = \langle X, HA^T \rangle + \langle Y , (A^{-1})^TK \rangle \\
        & = \langle (X,Y) , (HA^T,(A^T)^{-1}K) \rangle \\
        & = \langle (X,Y) , \theta_{A^T}(H,K) \rangle,
    \end{align*}
    whence $\theta_{A}^* = \theta_{A^T}$. Then we compute $d(\theta^{(X,Y)})_A B = (XB,-A^{-1}BA^{-1}Y)$ for all $A,B\in \mathrm{GL}(r,\mathbb{R})$, from which we deduce
    \begin{align*}
        \langle d(\theta^{(X,Y)})_{I_r} B , (X,Y) \rangle = \langle (XB,-BY) , (X,Y)\rangle = \langle XB ,X \rangle - \langle BY,Y \rangle = \langle B , X^TX - YY^T\rangle
    \end{align*}
    and $C(X,Y) = d(\theta^{(X,Y)})_{I_r}^* (X,Y) = X^TX - YY^T$.
\end{example}
\begin{example}[{Nonnegative matrix factorization}]
    \label{eg:nmf}
     Given a matrix $M \in \mathbb{R}^{m\times n}$, consider the function
    \begin{equation*}
    \begin{array}{cccc}
         f: & \mathbb{R}^{m\times r}\times \mathbb{R}^{r\times n} & \longrightarrow & \overline{\mathbb{R}} \\
        & (X,Y) & \longmapsto & \|XY-M\|_F^2+\delta_{\R^{m\times r}_+\times \R^{r\times n}_+}(X,Y).
    \end{array}
    \end{equation*}
  It is invariant under the action of the positive diagonal group $D_r(\R_+^*)$ on $\R^{m\times r}\times \R^{r\times n}$ defined by $\theta(D,(X,Y))=(XD,D^{-1}Y)$. Since $(X,Y)\mapsto \|XY-M\|_F^2$ is regular and $\R_+^{m\times r}\times \R_+^{r\times n}$ is regular, we can apply \cref{prop:orbital_projection_formula_boundary} and \cref{cor:conserved}. Then, as in \cref{eg:mf1_non_canonical}, for all $B\in D_r(\R)$ we have 
    \begin{align*}
        \langle d(\theta^{(X,Y)})_{I_r} B, (X,Y) \rangle   = \langle B , X^TX - YY^T\rangle=\langle B,P_{D_r(\R)}(X^TX-YY^T) \rangle.
    \end{align*}
 This yields a conserved quantity $C(X,Y)=P_{D_r(\R)}(X^TX-YY^T)$. This recovers the fact that $\mathrm{diag}(X^TX-YY^T)$ is conserved, as was recently proved using adhoc arguments \cite{josz2024prox}. 
\end{example}

\subsection{Conservative fields}
\label{subsec:Conservative fields}

We next give a conservation law for conservative fields. Recall that by \cite[Corollary 1]{bolte2020conservative} if $f$ admits a conservative field $D$, then for all $x\in \rint \dom  f$ we have $\overline{\partial}f(x)\subseteq \mathrm{co}D(x)$. Suppose $f$ is invariant under the natural action of a Lie subgroup $G$ of $\GL(n,\R)$. Applying the chain rule \cite[Exercise 10.7]{rockafellar2009variational} to $f(x) = f(gx)$ at $g^{-1}\overline{x}$ for any $g\in G$ yields 
\begin{equation}
\label{eq:invariant_chain_rule}
    \partial f(g^{-1}\overline{x}) = g^T\partial f(\overline{x}).    
\end{equation}
We thus impose such a condition on $D$ below.
\begin{corollary}
    \label{cor:conserved_conservative_field}
    Let $D:\R^n\rightrightarrows\R^n$ be a definable conservative field with definable potential $f:\R^n\to\overline{\R}$ such that $d(0,D)$ is locally bounded. Suppose $f$ is invariant under the natural action of a definable Lie subgroup $G$ of $\GL(n,\R)$. If for all $g\in G$ and $x\in \R^n$, $D(g^{-1}x)=g^{T}D(x)$, then 
    \begin{equation*}
         C(x) = P_{\mathrm{s}(\mathfrak{g})}(xx^T) 
    \end{equation*}
    is conserved along the differential inclusion
    \begin{equation*}
     \forae t>0  \quad x'(t) \in D(x(t)) .
    \end{equation*}
\end{corollary}
\begin{proof}
   Suppose the orbital projection formula in \cref{prop:projection_formula_CF} holds at $x\in G\overline{x}$, namely, $\langle D(x), \fg x\rangle=\{0\}$. For all $g\in G$, we have $$\langle D(g^{-1}x), \fg g^{-1}x\rangle=\langle g^{T}D(x), \fg g^{-1}x\rangle=\langle D(x),g\fg g^{-1}x \rangle=\langle D(x), \fg x \rangle=\{0\}.$$ Indeed, conjugation by $g$ induces an isomorphism of the Lie algebra \cite[Example 7.4(f)]{lee2012smooth}. Thus the projection formula actually holds everywhere in $G\overline{x}$. The remainder of the proof is identical to the one of \cref{cor:conserved}.
\end{proof}
\begin{example}[Deep neural network]
    \label{eg:nn}
    Consider the training loss of a fully connected neural network, where $n \in \mathbb{N}^*$ is the number of data points, $l\in \mathbb{N}^*$ the number of layers, $n_i$ the width of the $i$\textsuperscript{th} layer, $x_1,\hdots, x_n \in \mathbb{R}^{n_0}$ the input data, and $y_1,\hdots,y_n \in \mathbb{R}^{n_l}$ the output data. Let $f: \mathbb{R}^{n_1\times n_0} \times \cdots \times \mathbb{R}^{n_{l} \times n_{l-1}} \times \mathbb{R}^{n_1} \times \cdots \times \mathbb{R}^{n_l} \rightarrow \mathbb{R}$ be defined by
\begin{equation*}
    f(W_1,\hdots,W_l,b_1,\hdots,b_l) = \frac{1}{n}\sum_{i=1}^n |W_l \sigma( \cdots \sigma(W_1x_i+b_1)\cdots) + b_l - y_i|^2
\end{equation*}
where the activation function $\sigma:\mathbb{R}\rightarrow\mathbb{R}$ is positively homogeneous, i.e., $\sigma(tx) = t\sigma(x)$ for all $t>0$ and $x\in \mathbb{R}$. Examples include $\mathrm{ReLU}(x) =\max\{0,x\}$ and $\mathrm{ReLU}_a(x) =\max\{ax,x\}$ for some $0<a<1$. Notice that for any matrix $A$ and positive diagonal matrix $D$ of compatible sizes, we have $\sigma(DA) = D\sigma(A)$. Thus consider the action 
\begin{equation*}
    D (W,b) = (D_1W_1,D_2W_2D_1^{-1},\hdots,D_lW_lD_{l-1}^{-1},D_1b_1,D_2b_2,\hdots,D_{l-1}b_{l-1},b_l)
\end{equation*}
where $D_i\in D_{n_i}(\R_+^*)$. Since $f$ is locally Lipschitz, \cref{cor:conserved} yields the conserved quantity
\begin{equation*}
    c(W,b) = \begin{pmatrix}
        \mathrm{diag}(W_1W_1^T + b_1b_1^T - W_2^TW_2) \\
        \vdots \\
        \text{diag}(W_{l-1}W_{l-1}^T + b_{l-1}b_{l-1}^T - W_l^TW_l)
    \end{pmatrix}.
\end{equation*}
This corroborates the findings in \cite[
C.9.3]{zhao2023symmetries} which however assume $f$ to be smooth. 
 
When the activation $\sigma$ is bijective, as with $\mathrm{ReLU}_a$, the objective function $f$ actually admits more symmetries, but they do not necessarily yield other conserved quantities. With $l=2$ layers, $f$ is invariant under the action $A (W,b) = (\sigma^{-1}(A\sigma(W_1X)),W_2A^{-1},Ab_1,b_2)$ where $A\in\mathrm{GL}(r,\mathbb{R})$. It is locally linear around a point $(\overline{W},\overline{b})$ where all the entries of $\overline{W}_1$ are nonzero, but the symmetric part of the Lie algebra need not be larger (i.e., $\theta^*_g \neq \theta_{g^*}$ in the setting of Remark \ref{rem:conserved}). 

If one constraints the weights to be nonnegative, as in \cite{chorowski2014learning}, then the same quantity is conserved around points where none of the rows nor columns of the weights $W_i$ are equal to zero, nor are any of the entries of the bias terms $b_i$. This holds because the action is free at such points. Note that here we may not use regularity of $f$, which is violated due to the ReLU activation. Simply consider $f(W,b)=( W_2\mathrm{ReLU}(W_1+b_1)+b_2-1 )^2$ where $W_1,W_2,b_1,b_2 \in \mathbb{R}$. We have $f(\epsilon_1,1,\epsilon_1,1-\epsilon_2) = (2\mathrm{ReLU}(\epsilon_1) - \epsilon_2)^2$, which fails to be regular when $\epsilon_1 = 0$ and $\epsilon_2>0$. 

Suppose $D_i$ is a conservative field for $$f_i(W,b)=|W_l \sigma( \cdots \sigma(W_1x_i+b_1)\cdots) + b_l - y_i|^2.$$ 
If each $D_i$ is compatible with the action, then so is their sum as well as the convex hull of the sum. Thus, by \cref{cor:conserved_conservative_field}, $c(W,b)$ is conserved along the differential inclusion $$\forae t>0,~ (W'(t),b'(t))\in \mathrm{co}\left(\frac{1}{n}\sum_{i=1}^nD_i(W(t),b(t))\right).$$
\end{example}

\section{Discrete subgradient dynamics}
\label{sec:Discrete subgradient dynamics}

The orbital projection formulae and the conserved quantity derived above can be used to detect instability in discrete subgradient dynamics. The subgradient method with constant step size $\alpha>0$ for minimizing a function $f:\R\p n \to \R$ consists of choosing an initial point $x_0 \in \R\p n$ and generating a sequence of iterates $\{x_k\}_{k \in \N} \subseteq \R\p n$ such that
$$\forall k \in \N, ~~~ x_{k+1} \in x_k - \alpha \partial f(x_k).$$
Recall the following definition.
\begin{definition}{\cite[Definition 2.2]{josz2024sufficient}}
\label{def:strong_unstable}
A point $\overline{x} \in \mathbb{R}^n$ is a strongly unstable for $f:\mathbb{R}^n\rightarrow\mathbb{R}$ if there exists $\epsilon>0$ such that for all but finitely many constant step sizes $\alpha>0$ and for almost every initial point in $B_{\epsilon}(\overline{x})$, at least one of the iterates of the subgradient method does not belong to $B_{\epsilon}(\overline{x})$.
\end{definition}

We will employ the same geometric condition \cite[Definition 3.1 (i)]{mordukhovich2015higher} as in the previously known instability result \cite[Theorem 2.9]{josz2024sufficient}. It is stated below.

\begin{definition}
\label{def:submetric} A set-valued mapping $F:\mathbb{R}^n\rightrightarrows\mathbb{R}^m$ is metrically $\eta$-subregular at $(\overline{x},\overline{y}) \in \gph F$ with $\eta>0$ if there exist $\kappa,r>0$ such that $d(x,F^{-1}(\overline{y}))\leq \kappa (d(\overline{y},F(x))^\eta$ for all $x\in B_r(\overline{x})$.
\end{definition}

%A related and more palatable notion is as follows \cite{davis2022proximal}.

%\begin{definition}
%\label{def:sharp} A function $f:\mathbb{R}^n \rightarrow %\overline{\mathbb{R}}$ is sharp at $\overline{x} \in \mathbb{R}^n$ if % there exists $\kappa>0$ such that for all $x \in \mathbb{R}^n %\setminus (\partial f)^{-1}(0)$ near $\overline{x}$, we have %$d(0,\partial f(x)) \geq \kappa$.
%\end{definition}

%If $f$ is sharp at $\overline{x}$, then $\partial f$ is metrically %$\eta$-subregular at $(\overline{x},0)$ for any $\eta>1$.

% It will be helpful to determine the update rule of the conserved quantity.

% \begin{proposition}
%     \label{prop:update}
%     Let $f:\mathbb{R}^n\rightarrow \mathbb{R}$ be locally Lipschitz and invariant under the natural action of a Lie subgroup $G$ of $\mathrm{GL}(n,\mathbb{R})$, and let $C(x) = P_{\mathrm{s}(\mathfrak{g})}(xx^T)$. If $x \in \mathbb{R}^n$ and $\alpha\in \mathbb{R}$, then 
%     \begin{equation*}
%     \forall v\in\partial f(x),\quad     {C(x - \alpha v) = C(x) + \alpha^2 C(v)}.
%     \end{equation*}
% \end{proposition}
% \begin{proof}
% Since $\langle \partial f(x),\fg x   \rangle=\{0\}$ for all $v\in \partial f(x)$ by the orbital projection formula in \cref{prop:projection_formula_ll}, one has $P_{s(\fg)}(xv^T)=P_{s(\fg)}(vx^T)=0$. Hence $C(x-\alpha v)=P_{s(\fg)}[(x-\alpha v)(x-\alpha v)^T] 
%   =P_{s(\fg)} (xx^T) -\alpha  P_{s(\fg)} (xv^T) -\alpha  P_{s(\fg)} (vx^T)+\alpha^2 P_{s(\fg)} (vv^T) 
%       =P_{s(\fg)}  (xx^T)+\alpha^2 P_{s(\fg)} (vv^T)=C(x)+\alpha^2C(v)$.
% \end{proof}

A subset $S$ of $\mathbb{R}^n$ is disconnected if there exist nonempty disjoint relatively open (in $S$) sets $A$ and $B$ such that $S = A\cup B$. It is connected if it is not disconnected. A maximal connected subset of $S$ is called a connected component of $S$. Given $f:\R^n\to\R$ and $\ell \in {\mathbb{R}}$, let $[f = \ell] = \{ x \in \mathbb{R}^n : f(x) = \ell \}$ and $[f = \ell]_{\overline{x}}$ denote the connected component of $[f = \ell]$  containing a point $\overline{x}\in \mathbb{R}^n$. In contrast to \cite[Theorem 2.9]{josz2024sufficient}, in the result below, one need not propose a Chetaev function nor check the Verdier condition.

\begin{theorem}
\label{thm:unstable}
Let $f:\mathbb{R}^n\rightarrow \mathbb{R}$ be definable, locally Lipschitz, and invariant under the natural action of a Lie subgroup $G$ of $\mathrm{GL}(n,\mathbb{R})$, and let $C(x) = P_{\mathrm{s}(\mathfrak{g})}(xx^T)$. Fix a point $\overline{x}\in\R^n$. Suppose $[f = f(\overline{x})]_{\overline{x}} = G\overline{x}$ and $\partial f$ is metrically $\eta$-subregular at $(\overline{x},0)$ with $\eta>1$. If $$ \forall u \in \partial f(x), ~ \forall v \in \partial f(y), \quad    \langle C(u) , C(v) \rangle > 0$$
for all $x,y\in \mathbb{R}^n \setminus G\overline{x}$ near $\overline{x}$, then $\overline{x}$ is strongly unstable.
\end{theorem}
\begin{proof}
    If $\overline{x}$ is not a critical point, then $\overline{x}$ is strongly unstable by \cite[Theorem 1]{josz2023lyapunov}. If $\overline{x}$ is a critical point, then we check the assumptions of \cite[Theorem 2.9]{josz2024sufficient}, one by one. Since the orbit $[f = f(\overline{x})]_{\overline{x}} = G\overline{x}$ is a definable set, by \cref{fact:embedded_orbit} it is a $C\p 2$ embedded submanifold of $\mathbb{R}^n$, whose dimension is less than $n$ by metric $\eta$-subregularily. By recalling \eqref{eq:invariant_chain_rule}, we see that $G\overline{x}$ is composed of critical points of $f$. Since $f$ is locally Lipschitz, by \cref{prop:perturbed_projection_formula_ll} the perturbed projection formula holds near $\overline{x}$ (referred to as Verdier condition in \cite[Definition 2.6]{josz2024sufficient}). Suppose the inequality in the prompt holds in $B_r(\overline{x}) \setminus G\overline{x}$ for some $r>0$. Let $y \in B_r(\overline{x}) \setminus G\overline{x}$ and $w \in c(\partial f(y))$. By assumption, we have   
    \begin{equation*}
        \forall x \in B_r(\overline{x}) \setminus G\overline{x}, ~ \forall v \in \partial f(x), ~~~ \langle C(v) , w \rangle > 0.
    \end{equation*}
    Let $\mu:(0,r]\rightarrow \mathbb{R}$ be defined by
    \begin{equation*}
        \mu(t) = \inf \{ \langle C(v) , w \rangle : v\in \partial f(x), ~ x\in B_r(\overline{x}), ~ d(x,G\overline{x}) = t \}.
    \end{equation*}
    By \cite[Theorem 9.13]{rockafellar2009variational} and \cite[Proposition 3 p. 42]{aubin1984differential}, the set of feasible subgradients $v$ is compact. Since the objective function is positive for all feasible $v$, it follows that $\mu$ is positive on its domain. Since $\mu$ is definable, by the monotonicity theorem \cite[(1.2) p. 43]{van1998tame}, there exists $r_0\in(0,r)$ such that $\mu$ is $C^1$ and monotone on $(0,r_0)$. Thus there exists a $C^1$ definable function $\nu:[0,r_0]\rightarrow \mathbb{R}$ such that $\nu(0) = 0$, $\nu'(t)>0$, and $\mu(t) \geq \nu(t)$ for all $t \in (0,r_0]$. Indeed, if $\mu$ is strictly increasing it suffices to take $\nu$ to be equal to $\mu$ on $(0,r_0)$ up to an additive constant potentially, and possibly after reducing $r_0$; otherwise one can take a linear function. Observe that $\inf_{[r_0,r]} \mu$ is equal to
    \begin{equation*}
        \inf \{ \langle C(v) , w \rangle : s\in \partial f(x), ~ x\in B_r(\overline{x}), ~ r_0 \leq d(x,G\overline{x}) \leq r \}
    \end{equation*}
    and is therefore positive. Thus $\mu(t) \geq \nu(t)$ and $\nu'(t)>0$ for all $t\in(0,r]$ after taking an affine extension of $\nu$, and possibly after multiplying $\nu$ by a positive constant. If $x \in B_r(\overline{x}) \setminus G\overline{x}$, then $d(x,G\overline{x}) \leq |x-\overline{x}| \leq r$ and hence $\mu(d(x,G\overline{x})) \geq \nu(d(x,G\overline{x}))$. As a result,
    \begin{equation*}
        \forall x\in B_r(\overline{x}) \setminus G\overline{x}, ~~ \forall v \in \partial f(x), ~~~ \langle C(v) , w \rangle \geq \nu(d(x,G\overline{x})).
    \end{equation*}
    Let $\{x_k\}_{k\in \mathbb{N}} \subseteq B_r(\overline{x}) \setminus G\overline{x}$ be such that $x_{k+1} = x_k - \alpha v_k$ for some $\alpha>0$ and $v_k \in \partial f(x_k)$. By \cref{cor:conserved}, we have
    \begin{equation*}
        \langle C(x_{k+1}),w \rangle - \langle C(x_k),w \rangle = \alpha^2 \langle C(v_k),w \rangle \geq \alpha^2 \nu(d(x_k,G\overline{x})).
    \end{equation*}
    Since $\nu$ is strictly increasing, the continuous function $ \mathbb{R}^n\ni x \mapsto \langle C(x),w\rangle \in \mathbb{R}$ is a Chetaev function and $\overline{x}$ is strongly unstable by \cite[Theorem 2.9]{josz2024sufficient}.
\end{proof}

\cref{thm:unstable} is meant to demonstrate the potential of the topics explored in this paper for the study of algorithms. We leave the study of algorithms with symmetry for future work.

\section*{Acknowledgments}
I would like to thank the Co-Editor, Associate Editor, and reviewers for their valuable feedback and patience in the review process. I am greatly indebted to Wenqing Ouyang for his careful reading and numerous comments that helped improve the paper. I would also like to thank Théodore Fougereux for fruitful discussions.

\section{Appendix}

\begin{proof}[Proof of Fact \ref{fact:embedded_orbit}]
Let $\theta:G\times \R^n\to \R^n$ denote the action and $\overline{x}\in \R^n$. Since $G\overline{x}$ is definable, by \cite[Theorem 4.8]{van1996geometric} it admits a $C^k$ stratification with finitely many definable strata. Let $X$ denote a stratum of maximum dimension. Consider some $x\in X$. There exists a neighborhood $U$ of $x$ in $\R^n$ such that $G\overline{x}\cap U=X\cap U$, otherwise $X\cap \overline{Y}$ is nonempty for some stratum $Y$. In that case, $X\subseteq \overline{Y}\setminus Y$, yielding the contradiction $\dim X <\dim Y$ by \cite[Chapter 4, Theorem 1.8]{van1998tame}. By the definition of $C^k$ submanifolds in \cite[Chapter 1, Section 1]{hirsch2012differential} and \cite[Chapter 1, Theorem 3.1]{hirsch2012differential}, $G\overline{x}\cap U$ is a level set of a $C^k$ submersion $\Phi:U\to \R^{n-p}$ where $p=\dim X$ after possibly reducing $U$. Let $y=gx$ for some $g\in G$. The set $gU$ is a neighborhood of $y$ in $\R^n$ such that $Gx\cap gU=g(Gx\cap U)$ is a level set of the $C^k$ submersion $\Phi\circ \theta_{g^{-1}}:gU\to \R^{m-p}$. By \cite[Chapter 1, Section 1]{hirsch2012differential}, $G\overline{x}$ is a $p$-dimensional $C^k$ embedded submanifold of $\R^n$.
\end{proof}

\begin{proof}[Proof of Fact \ref{fact:orbit_rank}]
Let $x=\theta(h,\overline{x})$. We have $\theta^{(x)}(g)=\theta(g,x)=\theta(g,\theta(h,\overline{x}))=\theta(gh,\overline{x}) = \theta^{(\overline{x})}(gh)=\theta^{(\overline{x})}\circ R_h (g)$, where $R_h:G\ni g\mapsto gh\in G$. Since $R_h\circ R_h^{-1}=\mathrm{Id}_G$, the mapping $R_h$ is a diffeomorphism. By the chain rule \cite[Proposition 3.6(b)]{lee2012smooth}, we have $d(\theta^{(x)})_e=d(\theta^{(\overline{x})})_h\circ d(R_h)_e$. Thus $\rank d(\theta^{(x)})_e=\rank d(\theta^{(\overline{x})})_h=\rank d(\theta^{(\overline{x})})_e$. The second equality follows from the equivariant rank theorem \cite[Theorem 7.25]{lee2012smooth}. Indeed, $\theta^{(\overline{x})}:G \to M$ is equivariant with respect to the smooth transitive action of $G$ on itself and $\theta$. 
\end{proof}

\begin{proof}[Proof of Fact \ref{fact:injective}]
Since $\bar{A}$ is injective, so are nearby $A$ and $B$, namely $\mathrm{Ker}\bar{A} = \mathrm{Ker}A = \mathrm{Ker}B = \{0\}$. By definition
\begin{align*}
    d(\mathrm{Im}A,\mathrm{Im}B) & = \sup_{p \in \mathrm{Im}A, |p|=1} \inf_{q \in \mathrm{Im}B} |p-q|\\
    & = \sup_{x \in \mathbb{R}^n \setminus \mathrm{Ker}A} \inf_{q \in \mathrm{Im}B}  \left|\frac{Ax}{|Ax|} - q \right|  \\
    & \leq  \sup_{x \neq 0} \left|\frac{Ax}{|Ax|} - B\frac{x}{|Ax|} \right| \\
    & = \sup_{x \neq 0} \frac{|Ax - Bx|}{|Ax|} \\
    & = \sup_{x \neq 0} \frac{|Ax - Bx|}{|\bar{A}x|} \frac{|\bar{A}x|}{|Ax|} \\
    & = \sup_{x \neq 0} \frac{|(A-B)x|}{|\bar{A}x|} \left( \frac{|\bar{A}x|-|Ax|}{|Ax|} + 1 \right) \\
    & \leq  \sup_{x \neq 0} \frac{|(A-B)x|}{|\bar{A}x|} \left( \frac{|\bar{A}x-Ax|}{|Ax|} + 1 \right) \\
    & \leq  |A-B| \sup_{x \neq 0} \frac{|x|}{|\bar{A}x|} \left( \frac{|\bar{A}-A||x|}{|Ax|} + 1 \right) \\
    & = |A-B| \sup_{|x|=1} \frac{1}{|\bar{A}x|} \left(\frac{|\bar{A}-A|}{|Ax|} + 1 \right) \\
    & \leq  \frac{|A-B|}{\inf\limits_{|x|=1} |\bar{A}x|} \left(\frac{|\bar{A}-A|}{\inf\limits_{|x|=1} |Ax|} + 1 \right) \\
    & \leq  \frac{|A-B|}{\inf\limits_{|x|=1} |\bar{A}x|} \left( \frac{2|\bar{A}-A|}{\inf\limits_{|x|=1} |\bar{A}x|} + 1 \right) \\
    & \leq  \frac{2|A-B|}{\inf\limits_{|x|=1} |\bar{A}x|}
\end{align*}
where in the two last inequalities we assume that $|A-\bar{A}| \leq  \inf|_{x|=1} |\bar{A}x|/2$. This is done to ensure that $|\bar{A}y|-|Ay| \leq  |(\bar{A}-A)y| \leq  \inf|_{x|=1} |\bar{A}x|/2$ for all $|y|=1$, whence $|\bar{A}y| - \inf|_{x|=1} |\bar{A}x|/2 \leq  |Ay|$ and $\inf|_{y|=1} |\bar{A}y|/2 = \inf|_{y|=1}|\bar{A}y| - \inf|_{x|=1} |\bar{A}x|/2 \leq  \inf|_{y|=1}|Ay|$.
\end{proof}

\begin{proof}[Proof of Fact \ref{fact:embedded_tangent}]
By \cite[Proposition 5.16]{lee2012smooth}, there exists a neighborhood $U$ of $\overline{x}$ in $\mathbb{R}^n$ such that $U\cap M$ is a level set of a $C\p 2$ smooth submersion $\Phi:U\rightarrow \mathbb{R}^{n-k}$ where $k$ is the dimension of $M$. Let $D \Phi(x)$ denote the Jacobian of $\Phi$ at $x \in U$ defined by 
    \begin{equation*}
        D \Phi(x)_{ij} = \frac{\partial \Phi_i}{\partial x_j}(x).
    \end{equation*}
    Since $N_x M = \mathrm{Im} D \Phi(x)^T$ near $\overline{x}$ by \cite[Proposition 5.38]{lee2012smooth} and $D \Phi(\overline{x})^T$ is injective, \cref{fact:injective} ensures that 
    \begin{align*}
        d(T_xM,T_yM) & = d(N_x M,N_y M) \\
        & = d(\mathrm{Im} D \Phi(x)^T,\mathrm{Im} D \Phi(y)^T) \\
        & \leq  C|D \Phi(x)^T - D \Phi(y)^T| \\
        & = C\sup_{|u|=1} |(D \Phi(x) - D \Phi(y))^Tu| \\
        & = C\sup_{|u|=1} \sqrt{\sum_{i=1}^n \langle \nabla \Phi_i(x) - \nabla \Phi_i(y), u\rangle^2} \\
        & \leq  C \sqrt{\sum_{i=1}^n |\nabla \Phi_i(x) - \nabla \Phi_i(y)|^2} \\
        & \leq  C \sqrt{\sum_{i=1}^n L^2|x - y|^2} \\
        & \leq  CL\sqrt{n}|x-y|
    \end{align*}
    for some $C>0$. The existence of a constant $L>0$ is due to the mean value theorem and the fact that $D\Phi$ is $C\p 1$ \cite[Theorem 5.19]{rudin1964principles} (see also \cite[Théorème C.12]{gilbert2021fragments}).
\end{proof}

\begin{proof}[Proof of Fact \ref{fact:verdier}]
Let $\cX$ be a Verdier stratification of $\gph f$. We seek to show that $\widetilde \cX=\{\pi(X):~X\in \cX\}$ is a Verdier stratification of $\dom  f$,  where $\pi:\R^{n+1}\to \R^n$ is the projection onto the first $n$ variables. 
        
        First, $\widetilde \cX$ is a locally finite partition of $\dom f$. Let $\cX = \{X_i\}_{i \in I}$ for some index set $I$. Since $\cX$ is a partition of $\gph f$, $\bigcup_i X_i = \gph f$. Thus $\bigcup_i \pi(X_i) = \pi(\bigcup_i X_i) = \pi( \gph f) = \dom f$. Let $X,Y \in \cX$ such that $\pi(X) \cap \pi(Y) \neq \emptyset$. There exist $(x,\alpha) \in X$ and $(y,\beta) \in Y$ such that $\pi(x,\alpha) = \pi(y,\beta)$, i.e., $x=y$. But since $X,Y \subseteq \gph f$, $f(x) = \alpha$ and $f(y) = \beta$, so that $\alpha = \beta$. Hence $X \cap Y \neq \emptyset$ and thus $X = Y$. In particular, $\pi(X) = \pi(Y)$. Next, consider a point $\overline{x} \in\dom f$. Let $U$ be a neighborhood of $(\overline{x},f(\overline{x})) \in \gph f$ in $\R\p{n+1}$ such that $\{ i \in I : U \cap X_i \neq \emptyset\}$ is finite. Since $f$ is continuous on $\dom f$, $V = (\mathrm{Id}_{\R\p n}, f)\p{-1}(U)$ is open in $\dom f$. Since it contains $\overline{x}$, it is a neighborhood of $\overline{x}$ in $\dom f$. Suppose $V \cap \pi(X_i) \neq \emptyset$ for some $i\in I$. Then there exist $x \in V$ and $(y,f(y)) \in X_i$ such that $x=\pi(y,f(y))=y$. Thus $(x,f(x)) \in U \cap X_i$. As a result, $\{ i \in I : V \cap \pi(X_i) \neq \emptyset\} = \{ i \in I : U \cap X_i \neq \emptyset\}$ is finite.

        Second, if $X \in \cX$, then $\pi(X)$ is a $C\p k$ embedded submanifold of $\R\p n$. Indeed, the restriction $\pi|_{X} = \pi \circ \iota_X$ is a $C\p k$ smooth embedding, as we next show, implying that it is diffeomorphic onto its image and that its image is embedded \cite[Proposition 5.2]{lee2012smooth}. It is injective since if $\pi|_{X}(x,f(x)) = \pi|_{X}(y,f(y))$, then $x=y$ and $(x,f(x)) = (y,f(y))$. The restriction to the codomain $\widetilde{\pi|_{X}}:X\to \pi(X)$ remains continuous when equipping $\pi(X)$ with the subspace topology \cite[Corollary 3.10(b)]{lee2010introduction}. The inclusion map $\iota_X$ is a $C\p k$ smooth embedding since $X$ is a $C\p k$ embedded submanifold of $\R\p{n+1}$ \cite[Theorem 5.27]{lee2012smooth}. In particular, $\iota_X$ is homeomorphic onto its image, so it is an open map. Since $\pi$ is a smooth submersion, by \cite[Proposition 4.28]{lee2012smooth} it is also open. Thus the composition $\widetilde{\pi|_{X}}$ is open and $\pi|_{X}$ is a topological embedding. It remains to prove that $\pi|_{X}$ is a smooth immersion. Fix $p\in X$. By the chain rule \cite[Proposition 3.6]{lee2012smooth} an linearity of $\pi$, $d(\pi|_{X})_{p} = d\pi_{p} \circ d(\iota_X)_{p} = \pi \circ d(\iota_X)_{p}$. Let $v \in T_{p} X$ be such that $d(\pi|_{X})_{p}(v) = 0$. Set $(w,\ell) = d(\iota_X)_{p}(v) \in \R\p{n} \times \R$ and observe that $\pi(w,\ell) = w = 0$. By \cite[Proposition 5.35]{lee2012smooth}, there exist a smooth curve $\gamma:J\to \R\p{n+1}$, denoted $\gamma(t) = (x(t),\alpha(t))$, with $0\in J$, $\gamma(0) = p = (\overline{x},f(\overline{x}))$, and $\gamma'(0) = (0,\ell)$. Since $f$ is locally Lipschitz on its domain, there exist $L>0$ and a neighborhood $J'$ of $0$ in $J$ such that $|\alpha(t)-\alpha(0)| = |f(x(t))-f(\overline{x})|\leq L|x(t)-\overline{x}|$ for all $t\in J'$. Dividing by $|t| \neq 0$ and passing to the limit yields $|\ell| \leq L \times 0 = 0$. To sum up, $(w,\ell) = 0$ and since $\iota_X$ is a smooth immersion, $v = 0$.
        %If $t\neq 0$, then $(0,\hdots,0,1) \in d(\iota_X)_{p}(T_{p} X) \cong T_{p} X$, which is forbidden since $\cX$ is a nonvertical stratification. 

        Third, let $X,Y \in \cX$ be such that $\pi(X) \cap \overline{\pi(Y)} \neq \emptyset$. There exist $(x,\alpha) \in X$ and $(y_k,\beta_k) \in Y$ such that $\pi(y_k,\beta_k) \to \pi(x,\alpha)$, i.e., $y_k \to x$. Since $f$ is continuous on its domain, $\beta_k = f(y_k) \to f(x) = \alpha$. Thus $X \cap \overline{Y} \neq \emptyset$. As a result, $X \subseteq \overline{Y}$ and $\pi(X) \subseteq \pi(\overline{Y}) \subseteq \overline{\pi(Y)}$, where the second inclusion holds because $\pi$ is continuous.

        Fourth, let $X,Y \in \cX$ be such that $\pi(X) \subseteq \overline{\pi(Y)} \setminus \pi(Y)$ and $\overline{x} \in \pi(X)$. Since $\pi(X) \neq \emptyset$, we have $\pi(X) \cap \overline{\pi(Y)}\neq \emptyset$. From the previous paragraph, we know that $X \subseteq \overline{Y}$, but $X \neq Y$, so $(\overline{x},f(\overline{x}))\in X\subseteq \overline{Y}\setminus Y$. Using the Verdier condition, we find that
        \begin{subequations}
            \begin{align}
            d(T_x \pi(X),T_y \pi(Y)) & = d(\pi(T_{(x,f(x))}X),\pi(T_{(y,f(y))}Y)) \label{eq:verdier_a} \\
            & \leq \sqrt{1+L\p 2} \hspace{.4mm} d(T_{(x,f(x))}X,T_{(y,f(y))}Y) \label{eq:verdier_b} \\
            & \leq \sqrt{1+L\p 2} C |(x,f(x))-(y,f(y))| \label{eq:verdier_c}\\
            & \leq (1+L\p 2) C|x-y| \label{eq:verdier_d}
        \end{align}
        \end{subequations}
        for $x\in \pi(X)$ and $y\in \pi(Y)$ near $\overline{x}$. Indeed, \eqref{eq:verdier_a} is due to $\pi(T_{(x,f(x))}X) = T_x \pi(X)$. This equality holds because as a smooth embedding, $\widetilde{\pi|_{X}}$ is diffeomorphic onto its image \cite[Proposition 5.2]{lee2012smooth}, whence $d(\widetilde{\pi|_{X}})_p (T_pX) = T_p\pi(X)$ for all $p\in X$. Also, $\pi \circ \iota_X = \iota_{\pi(X)} \circ \widetilde{\pi|_{X}}$, so that $\pi \circ d(\iota_X)_p = d(\iota_{\pi(X)})_p \circ d(\widetilde{\pi|_{X}})_p$. With the usual identifications, it follows that $\pi(v) = d(\widetilde{\pi|_{X}})_p(v)$ for all $v \in T_p X$. To see why \eqref{eq:verdier_b} holds, note that the supremum is reached in the definition of the distance. Thus there exists $(v,\alpha) \in T_{(x,f(x))}X$ with $|\pi(v,\alpha)|=|v|=1$ such that for all $(w,\beta) \in T_{(y,f(y))}Y$, we have 
        \begin{align*}
            d(\pi(T_{(x,f(x))}X),\pi(T_{(y,f(y))}Y)) & = d(v,\pi(T_{(y,f(y))}Y)) \\
            & \leq |v-w| \\
            & \leq |(v,\alpha)-(w,\beta)| \\
            & \leq |(v,\alpha)| \left|\frac{(v,\alpha)}{|(v,\alpha)|}-\frac{(w,\beta)}{|(v,\alpha)|}\right|.
        \end{align*}
        and thus 
        $$d(v,\pi(T_{(y,f(y))}Y)) \leq |(v,\alpha)| \hspace{.6mm}  d\left(\frac{(v,\alpha)}{|(v,\alpha)|},T_{(y,f(y))}Y\right) \leq |(v,\alpha)| \hspace{.6mm} d(T_{(x,f(x))}X,T_{(y,f(y))}Y).$$
As seen previously, using a curve we find that $|\alpha| \leq L|v|= L$ for some $L>0$. Thus $|(v,\alpha)|\p 2 \leq 1 + L\p 2$. The Verdier condition implies the existence of a constant $C>0$ in \eqref{eq:verdier_c}. Finally, \eqref{eq:verdier_d} holds because $|f(x)-f(y)| \leq L |x-y|$ near $\overline{x}$.

Fifth, let $X \in \cX$. As shown above, $\widetilde{\pi|_{X}}:X \to \pi(X)$ is a diffeomorphism and $(\widetilde{\pi|_{X}})\p {-1}(x) = (x,f(x))$ for all $x\in \pi(X)$. Thus $f|_{\pi(X)} = \widehat{\pi} \circ (\widetilde{\pi|_{X}})\p {-1}$ is $C\p k$ as a composition of $C\p k$ functions, where $\widehat \pi : \R\p{n+1}\to \R$ is the projection onto the last variable.
\end{proof}

\begin{proof}[Proof of Fact \ref{fact:perturbed_formula}]
We first check that for all $(x,t)\in \epii f$, $T_{\epii f}(x,f(x))\subseteq T_{\epii f}(x,t)$. Polarizing then yields $\widehat N_{\epii f}(x,t) = T_{\epii f}(x,t)\p * \subseteq T_{\epii f}(x,f(x))\p * = \widehat N_{\epii f}(x,f(x))$ by \cite[Theorem 6.28]{rockafellar2009variational}. Let $w \in T_{\epii f}(x,f(x))$. There exist $\epii f\ni (x_k,t_k)\to (x,f(x))$ and $\tau_k\searrow 0$ such that $[(x_k,t_k) - (x,f(x))]/\tau_k \to w$. Thus $[(x_k,t_k+t-f(x)) - (x,t)]/\tau_k \to w$ and $\epii f\ni (x_k,t_k+t-f(x))\to (x,t)$, so that $w \in T_{\epii f}(x,t)$.

    Let $X\in \cX$, $(x,f(x))\in X$, and $v\in N_{\epii f}(x,f(x))$. There exist $\epii f \ni (x_k,t_k)\to (x,f(x))$ and $v_k\in \widehat{N}_{\epii f}(x_k,t_k) \subseteq \widehat{N}_{\epii f}(x_k,f(x_k))$ such that $v_k\to v$. 
    Since $f$ is lower semicontinous, $f(x)=\liminf_{k\to\infty} t_k \geq  \liminf_{k\to\infty} f(x_k)\geq f(x)$, so that $(x_k,f(x_k))\to (x,f(x))$. Since there are only finitely many strata in $\cX$ near $(x,f(x))$, by taking a subsequence if necessary, there exists $Y\in \cX$ such that $(x_k,f(x_k))\in Y$ for all $k\in \mathbb{N}$. If $X = Y$, then $X \ni (x_k,f(x_k))\to (x,f(x))$ and $v_k\in \widehat{N}_{\epii f}(x_k,f(x_k)) \subseteq \widehat{N}_X (x_k,f(x_k))$ with $v_k\to v$, using the fact that $X \subseteq \gph f \subseteq \epii f$. Thus $v \in N_X (x,f(x)) = N_{(x,f(x))} X$, where trhe equality holds because $X$ is an embedded submanifold of $\R\p n$. 
    If $X \neq Y$, then $X \cap Y = \emptyset$ and $X\cap \overline{Y} \neq \emptyset$, so that $X \subseteq \overline{Y}\setminus Y$. The Verdier condition yields 
    $$ d(\mathrm{span}(v_k),N_{(x,f(x))} X) \leq d(N_{(x_k,f(x_k))} Y, N_{(x,f(x))} X) = d(T_{(x,f(x))}X,T_{(x_k,f(x_k))}Y)  \to 0.$$
    We conclude that $v \in N_{(x,f(x))} X$ as in the proof of \cref{prop:orbital_projection_formula}.
\end{proof}

\section*{Funding and/or conflicts of interest/competing interests}

I have no conflicts of interests or competing interests to declare.

\bibliographystyle{abbrv}    
\bibliography{symmetry}
\end{document}